\documentclass[10pt]{amsart}
\usepackage{amssymb,amsmath,amsthm,amsfonts}
\usepackage{mathrsfs}
\usepackage[all]{xy}
\usepackage{graphicx}
\usepackage{pstricks,pst-plot,epsfig}
\usepackage{enumerate}

\makeatletter
\def\@fnsymbol#1{\ensuremath{\ifcase#1\or 1\or 2\fi}}
\makeatother

\def\sha{\mathcal{A}}

\def\shd{\mathcal{D}}

\def\shh{\mathcal{H}}
\def\shm{\mathcal{M}}
\def\sho{\mathcal{O}}
\def\shn{\mathcal{N}}

\def\shf{\mathcal{F}}

\newcommand{\isoto}[1][]{\xrightarrow[#1]%
{{\raisebox{-.6ex}[0ex][-.6ex]{$\mspace{1mu}\sim\mspace{2mu}$}}}}
\newcommand{\isofrom}[1][]{\xleftarrow[#1]%
{{\raisebox{-.6ex}[0ex][-.6ex]{$\mspace{1mu}\sim\mspace{2mu}$}}}}

\newcommand{\C}{\mathbb{C}}
\newcommand{\N}{\mathbb{N}}
\newcommand{\R}{\mathbb{R}}

\newcommand{\Z}{\mathbb{Z}}
\newtheorem{theorem}{Theorem}[section]
\newtheorem{proposition}[theorem]{Proposition}
\newtheorem{lemma}[theorem]{Lemma}
\newtheorem{corollary}[theorem]{Corollary}

\theoremstyle{definition}
\newtheorem{definition}[theorem]{Definition}

\newtheorem{example}[theorem]{Example}

\newtheorem{remark}[theorem]{Remark}

\newtheorem{assumption}[theorem]{Assumption}
\newtheorem{convention}[theorem]{Convention}
\newtheorem{condition}[theorem]{Condition}

\begin{document}
\author{Ana Rita Martins, Teresa Monteiro Fernandes, David Raimundo}
\title{Extension of functors for algebras of formal deformation}
\date{7/6/2012}

\thanks{The research of T.Monteiro Fernandes was supported by
Funda\c c{\~a}o para a Ci{\^e}ncia e Tecnologia, PEst OE/MAT/UI0209/2011.}

\thanks{The research of D. Raimundo was supported by
Funda\c c{\~a}o para a Ci{\^e}ncia e Tecnologia and Programa
Ci{\^e}ncia, Tecnologia e Inova\c c{\~a}o do Quadro
Comunit{\'a}rio de Apoio and by Funda\c c{\~a}o Calouste Gulbenkian (Programa Est\'imulo \`a Investiga\c c{\~a}o).}

\address{Ana Rita Martins\\ Faculdade de Engenharia da Universidade Cat\'{o}lica Portuguesa\\ Estrada Oct\'{a}vio Pato, 2635-631 Rio-de-Mouro
 Portugal\\ \texttt{ritamartins@fe.lisboa.ucp.pt}}

\address{Teresa Monteiro Fernandes\\ Centro de Matem\'atica e Aplica\c{c}\~{o}es Fundamentais e Departamento de Matem\' atica da FCUL, Complexo 2\\ 2 Avenida Prof. Gama Pinto, 1649-003, Lisboa
 Portugal\\ \texttt{tmf@ptmat.fc.ul.pt}}

\address{David Raimundo\\ Centro de Matem\'atica e Aplica\c{c}\~{o}es Fundamentais e Departamento de Matem\' atica da FCUL, Complexo 2\\ 2 Avenida Prof. Gama Pinto, 1649-003, Lisboa
 Portugal\\ \texttt{dsraimundo@fc.ul.pt}}


\begin{abstract}
Suppose we are given  complex manifolds $X$ and $Y$ together with substacks  $\mathcal{S}$ and $\mathcal{S}'$ of  modules over   algebras of formal deformation $\mathcal{A}$ on $X$ and $\mathcal{A}'$ on $Y$, respectively. 
Suppose also we are given a functor $\Phi$ from the category of open subsets of $X$ to the category of open subsets of $Y$ together with a  functor $F$ of prestacks from $\mathcal{S}$ to $\mathcal{S}'\circ\Phi$.  Then we give
conditions for the existence of a canonical functor, extension of $F$ to the category of
coherent
${\sha}$-modules such that the cohomology associated to the action of the formal parameter
$\hbar$ takes values in $\mathcal{S}$.  We give an explicit construction and prove that when the initial functor $F$ is exact on each open subset, so is its extension. Our construction permits to extend the functors of inverse image, Fourier transform, specialization and   microlocalization, nearby and vanishing cycles in the framework of $\shd[[\hbar]]$-modules. We also obtain a Cauchy-Kowalewskaia-Kashiwara theorem in the non-characteristic case as well as comparison theorems for regular holonomic $\shd[[\hbar]]$-modules and a coherency criterion for  proper direct images of good $\shd[[\hbar]]$-modules.
\end{abstract}
\maketitle

\tableofcontents

\section*{Introduction.}

 On a complex manifold $X$ we consider the sheaf $\shd_X$ of differential operators and  the sheaf $\shd_X[[\hbar]]$ (noted $\shd_X^{\hbar}$ for short) of formal differential operators on a parameter $\hbar$. For the main results on modules over  $\shd_X$  we refer to \cite{Ka2} and for those on modules over $\shd_X^{\hbar}$ we refer to \cite{KS2} and   to \cite{DGS}. The notion of algebras of formal deformation  and the main results we need here were obtained in \cite{KS2}.

Our first motivation  was to understand the behavior of a coherent $\shd_X^{\hbar}$-module near a submanifold $Y$.
The natural tool is to define conveniently  a functor of inverse image generalizing the $\shd$-module case.
Alternatively, one can also look for a generalization of the functor of specialization. Recall that inverse image  on the category of $\shd$-modules is not exact, unless we assume in addition that the objects are non-characteristic. On the other hand, specialization is an exact  functor on the Serre subcategory of specializable $\shd$-modules.

To treat inverse image turned out to be not too hard because one finds a natural candidate to play the role of transfer module as we shall see later. On the other hand,    $\shd_X^{\hbar}$ is not provided with a  natural equivalent to Kashiwara-Malgrange $V$-filtration and specialization is far from being  a mere copy of the $\shd$-module case so that the study of its properties takes an important place in this work.

For a given sheaf $\sha$ of coherent rings  one denotes by $\text{Mod}_{coh}(\sha)$ the abelian category of coherent left $\sha$-modules.  Let $\mathbb{K}$ be a unital commutative Noetherian ring with finite global dimension.

The general problem   then became the following:

Given two complex manifolds $X$ and $Y$, together   with two $\mathbb{K}$-algebras of formal deformation $\mathcal{A}$ on $X$, and $\mathcal{A}'$ on $Y$,
given a right exact (respectively exact) functor $F$ from a given full Serre subcategory $\mathcal{S}$ of $\text{Mod}_{coh}(\sha)$ to a given full subcategory $\mathcal{S}^\prime$ of
$\text{Mod}(\mathcal{A}')$, find   the natural  subcategory  containing $\mathcal{S}$ to which
 $F$ extends canonically   as a right exact (respectively exact) functor, let us say, $F^{\hbar}$.

For each $n\in \N_0$  and for each left $\sha$-module, consider the quotient $\shm_n=\shm/\hbar^{n+1}\shm$ and, for $n\leq k$, denote by $\rho_{k,n}$  the projection $\shm_n\to\shm_k$.
If one assumes that, for each $n$, $\shm_n\in\mathcal{S}$,  then the natural candidate $F^{\hbar}(\shm)$ will be the projective limit
\begin{equation}\label{E00}
\varprojlim_n F(\shm_n)
\end{equation}
of the associated projective system  $(F(\shm_n),F(\rho_{k,n}))_n$.
This construction will be the heart of our study.

To be rigorous, we will resort to the framework of stacks and the reason is that we will be interested in Serre subcategories whose objects are defined by local properties.
Recall that   stacks  provide the framework  where the notion of sheaves of categories takes a sense.  However, throughout this work, we only deal with the easiest example of stacks consisting precisely of sheaves of categories, since they are substacks of modules over a sheaf of $\mathbb{K}$-algebras and the restriction morphisms are nothing more than the usual restriction of sheaves to open subsets. In particular all these stacks  are $\mathbb{K}[[\hbar]]$-linear, where $\hbar$ denotes the  central formal parameter in each of the algebras.

Recall that one denotes by  $Op(X)$ the category of open subsets of $X$ where the morphisms are defined  by the inclusions.
Let $\mathscr{M}od(\sha)$ denote the stack $U\mapsto\text{Mod}(\sha|_U)$, $U\in Op(X)$. Given an abelian  substack $\mathcal{C}$ of $\mathscr{M}od(\sha)$, a full substack $\mathcal{C}'$ of $\mathcal{C}$ is said to be a full Serre substack if, for each $U\in Op(X)$, $\mathcal{C}'(U)$ is a full Serre subcategory of $\mathcal{C}(U)$.

Accordingly, in the sequel, $\mathcal{S}$ will denote a  full Serre substack  of the stack $\mathscr{M}od_{coh}(\sha): U\mapsto \text{Mod}_{coh}(\sha|_U)$.
For the sake of simplicity, and whenever there is no ambiguity, we shall often say that a coherent $\mathcal{A}|_U$-module defined on $U\in Op(X)$ belongs to $\mathcal{S}$ if it belongs to $\mathcal{S}(U)$.

Let us now outline the main result of this work:


Assume that we are given a  full Serre substack $\mathcal{S}$  of  $\mathscr{M}od_{coh}(\sha)$ and a  full Serre substack $\mathcal{S}'$ of  $\mathscr{M}od(\mathcal{A}')$.
Consider   the category  $\text{Mod}_{S}(\mathcal{A})$ of $\text{Mod}_{coh}(\sha)$ characterized  by the property that, for each $n$, the kernel and the cokernel of the action of $\hbar^{n+1}$ belong to  $\mathcal{S}(X)$.  Assume we are given a functor $\Phi$ from $Op(X)$ to $Op(Y)$ such that $\Phi(X)=Y$ and that $\Phi$ transforms any open covering of any $\Omega\in Op(X)$ on an open covering of $\Phi(\Omega)$.
Denote by $\Phi^\ast \mathcal{S}'$ the prestack $U\mapsto \Phi^\ast\mathcal{S}'=\mathcal{S}'(\Phi(U))$
and assume that we are  given a $\mathbb{K}[[\hbar]]$-linear functor of prestacks $F: \mathcal{S}\to \Phi^\ast\mathcal{S}'$.
This means, in particular, that for each pair  $V,U\in Op(X)$ with $V\subset U$ we have the following commutative diagram of functors
of categories whose vertical arrows are the restriction functors:
$$\begin{matrix}
\mathcal{S}(U) & \xrightarrow{F(U)} & \mathcal{S}'(\Phi(U)) \\
\downarrow &   & \downarrow \\
\mathcal{S}(V) & \xrightarrow{F(V)} & \mathcal{S}'(\Phi(V)).
\end{matrix}$$

Here we prove the following (Theorem \ref{T:1} below):
 If, for each  $U\in Op(X)$, $F(U):\mathcal{S}(U)\to \mathcal{S}'(\Phi(U))$ is right exact (respectively exact), then, under a condition on the vanishing of the cohomology for $\mathcal{S}'(V)$, with $V$ running on the objects of $Op(Y)$ (Condition \ref{A}), automatically fulfilled by coherent modules,  by (\ref{E00}) we obtain a  canonical functor $F^{\hbar}:\text{Mod}_\mathcal{S}(\mathcal{A})\to\text{Mod}(\mathcal{A}^\prime)$.    Moreover  $F^{\hbar}$  is right exact (respectively exact).

Namely, when $\mathcal{S}^\prime$ is a substack of  coherent $\mathcal{A}^\prime$-modules, then $F^\hbar$ takes values in $\text{Mod}_{\mathcal{S}^\prime}(\mathcal{A}^\prime)$. Moreover, if each $F(U)$ is exact, this extension is, in a certain sense, unique  up to isomorphism.

The term "canonical" means that our construction is indeed functorial in $\mathcal{S}, \mathcal{S}', \Phi$ and $F$  (cf. Remark \ref{P:nova}).

 After the preliminary results in Sections $1$ and $2$,  in Section $3$ we prove
 Theorem \ref{T:1} using the following key facts:
\begin{itemize}
\item{a right exact functor combined with the action of $\hbar^{n+1}$ transforms, for each $n$,
exact sequences of $\mathcal{A}$-modules into right exact
sequences of $\mathcal{A}'_n$-modules;}
\item{the exactness of $\Gamma(K, \cdot)$ for $K$ belonging to adequate basis of the topologies on the manifolds,  for the categories we consider;}
\item{the exactness  of projective limit on the category of projective systems
satisfying  Mittag-Lefler's condition.}
\end{itemize}

In Section $4$ we use Theorem \ref{T:1} to treat the case of
 $\mathcal{A}=\shd_X^{\hbar}$ (respectively $\mathcal{A}'=\shd_{Y}^{\hbar})$. The situation   is then  simpler because the modules $\shd_X^{\hbar}/\hbar^{n+1}\shd_X^{\hbar}$ are free over $\mathcal{A}_0\simeq \shd_X$, so technically we are bound to extend  a right exact functor $F$ defined on a Serre substack $\mathcal{S}$ of  coherent $\shd_X$-modules.

In this way we obtain a natural setting for the extensions of the functors of inverse image, direct image by a closed embedding, specialization,
nearby and vanishing cycles, Fourier transform and microlocalization for $\shd_X^{\hbar}$-modules, which are performed in Section \ref{S2}.
 Namely, in the case of the extended  inverse image functor for a morphism $f$, when we restrict to the Serre substack of non-characteristic modules, we prove a formal version of the Cauchy-Kowalewskaia-Kashiwara theorem (Theorem~\ref{CKKT}).
We also generalize the functor \textit{extraordinary} inverse image using the concept of duality introduced in \cite{DGS} and we prove in Proposition \ref{P:hol} and in  Corollary~\ref{C:hol} that the property of  holonomicity (as well as that of  regular holonomicity)
 is stable under  inverse image (respectively extraordinary inverse image).

Moreover, for the extension of the specialization, microlocalization, vanishing and nearby cycles functors, when we restrict to the category of regular holonomic $\shd_X^{\hbar}$-modules in the sense of \cite {DGS}, we obtain comparison theorems which are the formal version of the results proved by Kashiwara in \cite{Ka1}(Theorems~\ref{T:22} and~\ref{T60}, and Corollaries~\ref{comparison cycles} and~\ref{C61}).

Remaining natural questions are the (left) derivability of $F^{\hbar}$ as well as the extension of left exact functors.

For the first, a difficulty in constructing
an $F^\hbar$-projective subcategory  comes certainly from the  behavior of the (left exact) functor $\varprojlim$  for which we don't  have in general enough injectives. This functor also lacks good properties with respect to the usual operations in sheaf theory. Therefore, even if there  exists  an $F$-projective subcategory $\mathcal{P}$,  to our knowledge there is no canonical way of constructing an $F^\hbar$-projective subcategory.

In what concerns direct images, which are defined as the composition of derived functors, one being left exact, the other being right exact, our method no longer applies except in particular cases, such as closed embeddings.  But there is another way, since we show that
in the case of inverse  image the extended  functor  can be given using a convenient transfer module as in $\shd$-module theory.
Once having  available a good notion of transfer module, we  can also obtain a natural extension of the functor of direct image.
In this setting, we prove a formal version of the theorem of coherency of proper direct image for good $\shd^{\hbar}$-modules (Theorem~\ref{T:405}).

So, as a by product of our general construction together with the results of \cite{DGS},  the so called $Grothendieck's$ six operations are generalized to the formal case.

\thanks{\textbf{Aknowledgment} We thank Stephane Guillermou not only for pointing out inaccuracies but also for his useful suggestions.}

\vspace{2mm}

\textbf{Convention 1.} The results in the first  three sections, with few exceptions, hold in the more general context of Hausdorff locally compact topological spaces.
For simplicity, in view of our motivations,  we stay in the complex analytic setting.

\section{Review on modules over formal deformations.}

In this section we recall the basic material we need from ~\cite{KS2}.

Let $X$ be a complex manifold of finite dimension $d_X$.

Let $\mathbb{K}$ be a unital commutative Noetherian ring with finite global dimension. 
Recall that, for brevity, one denotes $\C[[\hbar]]$ by $\C^{\hbar}$.

Given a sheaf $\mathcal{R}$ of $\mathbb{K}$-algebras on $X$, we denote by $\text{Mod}(\mathcal{R})$ the category of left $\mathcal{R}$-modules, by $D(\mathcal{R})$ the derived category of $\text{Mod}(\mathcal{R})$ and by $D^\ast(\mathcal{R})$ ($\ast=+,-,b)$ the full triangulated subcategory of $D(\mathcal{R})$ consisting of objects with bounded from below (resp. bounded from above, resp. bounded) cohomology.

Recall that a full subcategory  $\mathcal{S}$ of an abelian category $\mathcal{C}$  is thick if for any exact sequence $$Y\to Y'\to X\to Z\to Z'$$ in $\mathcal{C}$ with $Y, Y', Z, Z'$ in $\mathcal{S}$, $X$ belongs to $\mathcal{S}$.

Equivalently, $\mathcal{S}$ is a full abelian subcategory such that,  given a short exact sequence $0\to X'\to X\to X''\to 0$ in $\mathcal{C}$, when two of the objects $X', X$ or $X''$ are in $\mathcal{S}$ then the third also belongs to $\mathcal{S}$.  If, moreover, $\mathcal{S}$ contains all subobjects and quotient objects of its objects,
then $\mathcal{S}$ is called a Serre subcategory.

 Let $\mathcal{S}$ be a full thick subcategory of $\mathcal{C}$ and let $D^{b}(\mathcal{C})$ be the bounded derived category of $\mathcal{C}$. One denotes by $D^b_{\mathcal{S}}(\mathcal{C})$ the full triangulated subcategory of $D^b(\mathcal{C})$ consisting of objects with cohomology in $\mathcal{S}$. In the cases listed below we recall classical abbreviations.

\begin{example}\label{example1}

\noindent
\begin{itemize}
\item{The subcategory $\text{Mod}_{coh}(\mathcal{R})$ of coherent modules over a coherent ring $\mathcal{R}$ is thick, and  the associated category  is denoted  by $D^b_{coh}(\mathcal{R})$.}

\item{The subcategory $\text{Mod}_{\R-c}(\mathbb{K}_X)$ of  $\R$-constructible sheaves of $\mathbb{K}$-modules is a thick subcategory of $\text{Mod}(\mathbb{K}_X)$ and the associated category is denoted by $D^b_{\R-c}(\mathbb{K}_X).$}
\end{itemize}
For a complex manifold $X$, the following are Serre subcategories of  $\text{Mod}_{coh}(\mathcal{\shd}_X)$:
\begin{itemize}
\item{The subcategory $\text{Mod}_{good}(\shd_X)$ of good $\shd_X$-modules  and  the associated category  is denoted  by $D^b_{good}(\shd_X)$.}
\item{The subcategory  $\text{Mod}_{hol}(\shd_X)$ of holonomic $\shd_X$-modules and the subcategory  $\text{Mod}_{rh}(\shd_X)$ of regular holonomic $\shd_X$-modules and the associated categories are respectively denoted by $D^b_{hol}(\shd_X)$ and  $D^b_{rh}(\shd_X).$}
\item{The subcategory $NC(f)$ of non-characteristic $\shd_X$-modules with respect to a given holomorphic function $f:Y \to X$, where $Y$ is another complex manifold and the associated category is denoted by $D^b_{NC(f)}(\shd_X)$.}
\item{The subcategory $\text{Mod}_{sp}(\shd_X)$ of $\shd_X$-modules specializable along a given submanifold $Y$ and the associated category  is denoted  $D^b_{sp}(\shd_X)$.}

\end{itemize}
\end{example}

Given a sheaf $\shm$ of $\mathbb{Z}_X[\hbar]$-modules, set $\shm_n= \shm/\hbar^{n+1}\shm$ and for $n\geq k$ let  $\rho_{k,n}:\shm_n\to \shm_k$ denote the canonical epimorphisms.
One says that $\shm$ is $\hbar$-torsion free if $\hbar: \shm\to \shm$ is injective and one says that $\shm$ is $\hbar$-complete if
the canonical morphism $\shm\to\underset{{n\geq 0}}\varprojlim \shm_n$ is an isomorphism.

A family $\mathcal{B}$ of compact subsets of $X$ is said to be a basis of compact subsets of $X$ if for any $x\in X$ and any open neighborhood $U$ of $x$, there exists $K\in\mathcal{B}$ such that $x\in Int(K)\subset U$.

In the following we shall consider a $\mathbb{K}$-algebra $\mathcal{A}$ on $X$ and a section $\hbar$ of $\mathcal{A}$ contained in the center of $\mathcal{A}$. Set $\mathcal{A}_0=\mathcal{A}/\hbar\mathcal{A}$.

Consider the following conditions:
\begin{enumerate}[{\rm (i)}]
\item $\mathcal{A}$ is $\hbar$-torsion free and is $\hbar$-complete,
\item $\mathcal{A}_0$ is a left Noetherian ring,
\item
there exists a basis $\mathcal{B}$ of open subsets of $X$ such that for any $U\in\mathcal{B}$
and any coherent $(\mathcal{A}_0\vert_U)$-module $\shf$ we have $H^n(U;\shf)=0$ for any $n>0$,
\item
there exists a basis $\mathcal{B}$ of compact subsets of $X$ and a prestack
$U\mapsto\text{Mod}_{good}(\mathcal{A}_0\vert_U)$ ($U$ open in $X$) such that
\begin{enumerate}[{\rm (a)}]
\item for any $K\in \mathcal{B}$ and an open subset $U$ such that $K\subset U$,
there exists $K'\in\mathcal{B}$ such that $K\subset Int(K')\subset K'\subset U$,
\item
$U\mapsto \text{Mod}_{good}(\mathcal{A}_0\vert_U)$ is a full subprestack of
$U\mapsto \text{Mod}_{coh}(\mathcal{A}_0\vert_U)$,
\item for an open subset $U$ and $\shm\in\text{Mod}_{coh}(\mathcal{A}_0\vert_U)$,
if $\shm\vert_V$ belongs to $\text{Mod}_{good}(\mathcal{A}_0\vert_V)$ for any
relatively compact open subset $V$ of $U$,
then $\shm$ belongs to $\text{Mod}_{good}(\mathcal{A}_0\vert_U)$,\label{cond:exh}
\item
for any $U$ open in $X$,
$\text{Mod}_{good}(\mathcal{A}_0\vert_U)$ is stable by subobjects (and hence, by quotients)
in $\text{Mod}_{coh}(\mathcal{A}_0\vert_U)$,
\item
for any $K\in\mathcal{B}$, any open set $U$ containing $K$, any $\shm\in\text{Mod}_{good}(\mathcal{A}_0\vert_U)$
and any $j>0$, one has $H^j(K;\shm)=0$,
\item for any $\shm\in\text{Mod}_{coh}(\mathcal{A}_0\vert_U)$, there exists an open covering
$U=\bigcup_iU_i$ such that
$\shm\vert_{U_i}\in\text{Mod}_{good}(\mathcal{A}_0\vert_{U_i})$,
\label{goodlocal}
\item $\mathcal{A}_0\in\text{Mod}_{good}(\mathcal{A}_0)$.\label{cond:good}
\end{enumerate}
\end{enumerate}

We shall say that $\mathcal{A}$ is an algebra of formal deformation if $\mathcal{A}$ and $\mathcal{A}_0$ satisfy either Assumption \ref{AssumptionA} or Assumption \ref{AssumptionB} below:\begin{assumption}\label{AssumptionA}
 $\mathcal{A}$ and $\mathcal{A}_0$ satisfy  conditions ${\rm (i)}$, ${\rm (ii)}$ and ${\rm (iii)}$ above.
\end{assumption}
\begin{assumption}\label{AssumptionB}
$\mathcal{A}$ and $\mathcal{A}_0$ satisfy  conditions
${\rm (i)}$, ${\rm (ii)}$ and ${\rm (iv)}$ above.
\end{assumption}

In particular, with Assumption~\ref{AssumptionA} or Assumption~\ref{AssumptionB}, $\mathcal{A}$ and $\mathcal{A}_n$ are left Noetherian rings, for every $n\geq 0$ (see Lemma 1.2.3 and Theorems 1.2.5 and 1.3.6 of ~\cite{KS2}).

One defines a right exact functor assigning
the object $\shm/\hbar^{n+1} \shm\in\text{Mod}(\mathcal{A}_n)$ to $\shm\in\text{Mod}(\mathcal{A})$.
Its left derived functor is given by:
\begin{eqnarray*}
gr_\hbar^n :D^b(\mathcal{A})&\to&D^b(\mathcal{A}_n),\\
\shm&\mapsto& \shm\stackrel{{L}}{\otimes}_{\mathcal{A}} \mathcal{A}_n.
\end{eqnarray*}

Recall that the functor $gr_\hbar^0$ was defined and studied in~\cite{KS2} and noted by $gr_\hbar$.

For $\shm\in\text{Mod}(\mathcal{A})$ one sets:
$$_{n}\shm= \ker(\shm\stackrel{\hbar^{n+1}}{\to}\shm).$$

Recall that a coherent $\mathcal{A}$-module   is a locally finitely generated $\mathcal{A}$-module $\shm$ such that, for any open subset $U\subset X$ and for each  locally finitely generated submodule $\shm'$ of $\shm|_U$,  locally  $\shm'$ admits a   finite free presentation.

If $\shm$ is a coherent $\mathcal{A}$-module then $_n\shm$ and $\shm_n$
are coherent $\mathcal{A}_n$-modules.

Recall that, for each $n\geq 0$, the category Mod$(\mathcal{A}_n)$ and the full subcategory of Mod$(\mathcal{A})$ whose objects are those $\shm$ such that $\hbar^{n+1}\shm\simeq 0$ are equivalent. Moreover:

\begin{lemma}[\cite{KS2}, Lemma 1.2.3]
Let $n\geq 0$. An $\mathcal{A}_n$-module $\shm$ is coherent as an $\mathcal{A}_n$-module if and only if it is so as an $\mathcal{A}$-module.
\end{lemma}

Suppose that the property (iv) in Assumption~\ref{AssumptionB} holds.
Denote by $D^b_{good}(\mathcal{A}_0)$ the full triangulated subcategory of $D^b(\mathcal{A}_0)$ consisting of objects with cohomology in   $\text{Mod}_{good}(\mathcal{A}_0)$. One says that $\shm\in D^b_{coh}(\mathcal{A})$ is good if $gr_\hbar(\shm) \in D^b_{good}(\mathcal{A}_0)$.
One denotes by $D^b_{good}(\mathcal{A})$ the full subcategory of $D^b_{coh}(\mathcal{A})$ consisting of good objects.

\begin{theorem}[\cite{KS2}, Theorem 1.3.6]\label{T2}
For any good $\mathcal{A}$-module $\shm$ and any $K\in\mathcal{B}$, we have $H^j(K; \shm)=0,$ for any $j>0$.
\end{theorem}

\begin{theorem}[\cite{KS2}, Theorem 1.3.6]\label{T1}
An $\mathcal{A}$-module $\shm$ is coherent if and only if it is $\hbar$-complete and $\hbar^n\shm/\hbar^{n+1}\shm$ is a coherent $\mathcal{A}_0$-module for any $n\geq 0$.
\end{theorem}

For a  sheaf $\mathcal{R}$ of $\Z[\hbar]$-algebras,
set $\mathcal{R}^{\text{loc}}:=\mathcal{R}\otimes_{\Z_X[\hbar]}\Z_X[\hbar,\hbar^{-1}]$.
Following ~\cite{KS2}, $\shm\in D^b(\mathcal{R})$ is said to be a cohomologically $\hbar$-complete object if $R\shh\text{om}_{\mathcal{R}}(\mathcal{R}^{loc},\shm)\simeq R\shh\text{om}_{\Z_X[\hbar]}(\Z_X[\hbar,\hbar^{-1}],\shm)=0.$ We shall use for short the symbol $c\hbar c$ to distinguish cohomologically $\hbar$-complete objects.

\remark The category of $c \hbar c$ objects is a full triangulated subcategory of $D^b(\mathcal{R})$.  Namely, if in a distinguished triangle two of the terms are $c\hbar c$ the third is also $c\hbar c$.

Recall  that any $\shm\in D^b_{coh}(\mathcal{\sha})$ is $c \hbar c$.

For convenience, we denote by $\mathcal{C}$ the subcategory of $c\hbar c$-modules of $\text{Mod}(\Z_X[\hbar])$.

\begin{lemma}\label{chc}
$\mathcal{C}$ is a full abelian thick  subcategory of $\text{Mod}(\Z_X[\hbar])$.
\end{lemma}

\begin{proof}
By the remark above it remains to prove that $\mathcal{C}$ is closed under kernels and cokernels.
Given a morphism $f:A\to B$ in $\mathcal{C}$,   the mapping cone $M(f)$ is $c\hbar  c$ in $D^b(\Z_X[\hbar])$ so from the distinguished triangle $$\text{ker}f[1]\to M(f)\to \text{coker}f\overset{+1}{\to}$$
we derive a distinguished triangle:$$R\shh\text{om}_{\Z_X[\hbar]}(\Z_X[\hbar,\hbar^{-1}],\text{ker} f[1])\to
R\shh\text{om}_{\Z_X[\hbar]}(\Z_X[\hbar,\hbar^{-1}], M(f))\to$$ $$\to
R\shh\text{om}_{\Z_X[\hbar]}(\Z_X[\hbar,\hbar^{-1}],\text{coker}
f)\xrightarrow{+1}.$$
Besides, $\Z[\hbar,\hbar^{-1}]$ is a $\Z[\hbar]$-module
with projective dimension $\leq 1$, so
$$R^j\shh\text{om}_{\Z_X[\hbar]}(\Z_X[\hbar,\hbar^{-1}],\text{coker} f)=0, \text{for} \ j\neq 0,1$$
$$R^j\shh\text{om}_{\Z_X[\hbar]}(\Z_X[\hbar,\hbar^{-1}],\text{ker} f)=0, \text{for} \ j\neq 0,1.$$
The result follows from the long exact sequence attached to the preceding triangle.
\end{proof}
\begin{theorem}[\cite{KS2}, Theorem 1.6.4]\label{T123}
Let $\shm\in D^b(\mathcal{A})$ and assume that $\shm$ is $c \hbar c$ and $gr_\hbar(\shm)$ is an object of $D^b_{coh}(\mathcal{A}_0)$. Then, $\shm$ is an object of $D^b_{coh}(\mathcal{A})$ and we have the isomorphisms
\begin{equation}\label{cohom}
H^i(\shm)\simeq \varprojlim_{n\geq 0} H^i(gr_\hbar^n(\shm)).
\end{equation}
\end{theorem}

\begin{theorem}[\cite{KS2},Theorem 1.6.6]\label{T124} Assume that $\mathcal{A}^{op}/\hbar\mathcal{A}^{op}$ is a Noetherian ring.
Let $\shm$ be a $c \hbar c$ $\mathcal{A}$-module with no $\hbar$-torsion and such that $\shm/\hbar\shm$ is a flat $\mathcal{A}_0$-module.
Then $\shm$ is a flat $\sha$-module.
\end{theorem}

\begin{proposition}[\cite{KS2}, Corollary 1.5.9]\label{P121}
The functor $gr_{\hbar}$ is conservative in the category of $c \hbar c$ objects. In particular  it is conservative in $D^b_{\R-c}(\C^\hbar_X)\to D^b_{\R-c}(\C_X)$ and in $D^b_{coh}(\mathcal{A})\to D^b_{coh}(\mathcal{A}_0)$.
\end{proposition}

\begin{proposition}[\cite{KS2}, Corollary 1.5.7]\label{P122}
Assume that $\shm\in\text{Mod}(\mathcal{A})$ is $\hbar$-complete and $\hbar$-torsion free. Assume that there exists a basis $\mathcal{B}$ of open (respectively of compact)  subsets $\Omega$ such that $H^i(\Omega; \shm)=0$ for $i>0$. Then $\shm$ is $c\hbar c$.
\end{proposition}

\begin{proposition}[\cite{KS2}, Proposition 1.5.10]\label{P1}
If $\shm\in\text{D}^b(\mathcal{A})$ is $c \hbar c$, then
$R\mathcal{H}{om}_{\mathcal{A}}(\shn ,\shm)$  is $c \hbar c$, for any $\shn\in
D(\mathcal{A})$.
\end{proposition}

\begin{proposition}[\cite{KS2}, Proposition 1.5.12]\label{P125}
Let $f:X\to Y$ be a morphism of complex manifolds, and suppose that $\shm\in D^b(\mathcal{A})$ is $c\hbar c$.
Then, $R f_\ast \shm$ is also $c\hbar c$.
\end{proposition}

Let now $f:Y\to X$ be a morphism of complex manifolds and let us consider the canonical morphisms:
\begin{center}
$f_\pi:X\times_Y T^*Y\rightarrow T^*Y$ and $f_d:X\times_Y
T^*Y\rightarrow T^*X$.
\end{center}

Recall that $f$ is said to be non-characteristic  for an object $F\in D^b(\mathbb{K}_X)$  if $$f_\pi^{-1}(SS(F))\cap \ker f_d\subset Y\times_ X T_X^*X,$$ where $SS(F)$ denotes the microsupport of $F$. We refer to \cite{KS1} for a detailed study of the notion of microsupport.

We shall also need in addition the result below:
\begin{proposition}\label{P:500}
Let $f: Y\to X$ be a morphism of complex manifolds.
\begin{itemize}
\item[(i)] Assume that a $c \hbar c$ object $\shm\in\text{D}^b(\Z_X[\hbar])$ is non characteristic for $f$. Then $f^{-1}\shm $ is $c \hbar c$;
 \item[(ii)] For every $\shm\in\text{D}^b( \Z_X[\hbar])$, one has $gr_\hbar (f^{-1}\shm)\simeq f^{-1}gr_\hbar (\shm)$.
\end{itemize}
\end{proposition}
\begin{proof}

(i) By\cite [Prop. 5.4.13 (ii)]{KS1}, the result follows from the isomorphism $$R\mathcal{H}{om}_{\mathbb{Z}_Y[\hbar]}(\mathbb{Z}_Y[\hbar, \hbar^{-1}], f^!\shm)\simeq f^!R\mathcal{H}{om}_{\mathbb{Z}_X[\hbar]}(\mathbb{Z}_X[\hbar, \hbar^{-1}], \shm). $$

(ii) is clear.
\end{proof}

For $\shm \in \text{Mod}(\mathcal{A})$ one denotes by $\shm_{\hbar-tor}$ the submodule of $\shm$
consisting of sections locally annihilated by some power of $\hbar$ and
by $\shm_{\hbar-tf}$ the quotient $\shm/\shm_{\hbar-tor}$.
Thus the following sequence:
\begin{equation}\label{E:111}
0\to\shm_{\hbar-tor}\to\shm\to\shm_{\hbar-tf}\to 0.
\end{equation}
is exact.

$\shm\in\text{Mod}(\mathcal{A})$ is said to be an $\hbar$-torsion module if $\shm_{\hbar-tor}\simeq\shm$
and $\shm$ is $\hbar$-torsion free if and only if $\shm\simeq\shm_{\hbar-tf}$.
In particular, for each $n\geq 0$, $\shm_n$ is an $\hbar$-torsion module since $\hbar^{n+1}\shm_n=0$.

Note that $\shm_{\hbar-tor}$ is also the increasing union of the $_n\shm$'s.
If $\shm$ is  coherent, the  family $\left\{_n\shm\right\}_n$ is locally stationary, so
locally there exists $N\geq 1$ such that $\hbar^N\shm_{\hbar-tor}=0$ and  both $\shm_{\hbar-tor}$ and $\shm_{\hbar-tf}$ are coherent $\mathcal{A}$-modules.

In particular, an $\hbar$-torsion $\mathcal{A}$-module
 is coherent as an $\mathcal{A}$-module if and only if, locally,  it is  coherent as an $\mathcal{A}_n$-module for $n$ big enough.

If $\shm$ is a coherent $\mathcal{A}$-module, then  each $\shm_n$ is coherent
as an $\mathcal{A}$-module, thus as an $\mathcal{A}_n$-module.

\begin{lemma}\label{L1}
Let $0\to\shm^\prime\to\shm\to\shm^{\prime\prime}\to 0$
be an exact sequence in $\text{Mod}(\mathcal{A})$ and
suppose that $\shm^{\prime\prime}$ is $\hbar$-torsion free.
Then, for each $n\geq 0$, the associated sequence of $\mathcal{A}_n$-modules:
\begin{eqnarray}\label{eq8}
0\to\shm^\prime_n\to\shm_n\to\shm^{\prime\prime}_n\to 0.
\end{eqnarray}
is exact.
\end{lemma}

\begin{proof}
For each $n\geq 0$, applying $gr_\hbar^n$ to $0\to\shm^\prime\to\shm\to\shm^{\prime\prime}\to 0$, we deduce the long exact sequence $$0\to _n\shm'\to_n\shm\to_n\shm''\to\shm'_n\to\shm_n\to\shm''_n\to 0.$$ By assumption $_n\shm''=0$ and the result follows.
\end{proof}

\begin{corollary}\label{C1}
Let $\shm$ be an $\mathcal{A}$-module.
Then, for each $n\geq 0$, the following sequence  is exact:\begin{equation}\label{eqL7}
0\to{\shm_{\hbar-tor}}_n\to \shm_n \to {\shm_{\hbar-tf}}_n\to 0.
\end{equation}
\end{corollary}

Let $\shm\in \text{Mod}(\mathcal{A})$, let $n'\geq n-k$ and denote by  $\overline{\hbar^k}:\shm_{n'}\to\shm_n$ the morphism defined by the multiplication by $\hbar^k$.
Observe that the action of  $\hbar^k$ in $\shm_n$ coincides with the composition of the chain of morphisms
$$\shm_n\xrightarrow{\overline{\hbar^k}} \shm_{n+k}\xrightarrow{\rho_{n,n+k}} \shm_n.$$

 \begin{lemma}\label{L:19}
For each $n\geq k\geq 1$ and each $n'\geq n-k$ one has an exact sequence:
\begin{equation}\label{eq6}
\shm_{n'}\xrightarrow{\overline{\hbar^k}}\shm_{n}\xrightarrow{\rho_{k-1,n}}\shm_{k-1}\to 0.
\end{equation}

 \end{lemma}

\begin{proof}
Clearly  $\ker(\rho_{k-1,n})=\hbar^k\shm/\hbar^{n+1}\shm=\overline{\hbar^k}(\shm_{n'})$.
\end{proof}

\begin{lemma}\label{L:19b}
Let $\shm$ be an $\hbar$-complete $\sha$-module.
Then $\shm$ is $\hbar$-torsion free if and only if for every $n\geq 0$ the sequence below is exact:
\begin{equation}\label{eq110}
0\to\shm_{0}\xrightarrow{\overline{\hbar^n}}\shm_n\xrightarrow{\hbar}\shm_{n}.
\end{equation}
\end{lemma}
\begin{proof}
If $\shm$ $\hbar$-torsion free, (\ref{eq110}) is clearly exact  since, for $m, m'\in\shm$, the equality $\hbar^n m=\hbar^{n+1}m'$ entails $m=\hbar m'$.

Conversely, assume that for every $n\geq 0$ we have the exact sequence~\eqref{eq110}.
Thus, given $(v_n)_n\in \shm$ such that $\hbar v_n=0, \, \forall n$,  it follows that $v_n=\overline{{\hbar}^n} u_{0n}$ for some (unique) $u_{0n}\in \shm_0$ and we may choose  $u_n\in\shm_n$ such that  $v_n=h^nu_n$, $\forall n$.
On the other hand $v_n=\rho_{nn'}(v_{n'}), \forall n'\geq n$, hence $v_n=\rho_{nn'}( {\hbar}^{n'}u_{n'})=\hbar^{n'}\rho_{nn'}(u_{n'})=0$
since we may take $n'\geq n+1$.
\end{proof}

Given  a  full substack $\mathscr{C}: U\mapsto \mathscr{C}(U)$ of $\mathscr{M}od(\sha)$ of abelian subcategories, we shall consider the following condition defining a full Serre substack $\mathscr{S}$ of  $\mathscr{C}$:
 \begin{condition}\label{A}
For each $U$, $\shm$ belongs to $\mathscr{S}(U)$ if and only if, for each $x\in U$, there exists a neighborhood $V\subset U$ of $x$ such that, for any submodule $\shn$ of $\shm$ belonging to $\mathscr{C}(V)$
   (and hence for any quotient $\mathcal{N}$ of $\shm$ belonging to $\mathscr{C}(V)$),
 if $K\in\mathcal{B}$ is contained in $V$, then
 \begin{equation}\label{E:T23}
 H^j(K;\shn)=0,\,\, \text{for any}\, j>0.
 \end{equation}
\end{condition}

In particular, if $\mathscr{C}=\mathscr{M}od_{coh}(\sha)$, then $\mathscr{S}=\mathscr{C}$.

\begin{lemma}\label{L:Dav}
Let $\shm\in\text{Mod}(\sha)$  and suppose that $\shm_0$ belongs to $\mathscr{S}(X)$. Let $(V_i)_i$ be an open covering of $X$ where Condition \ref{A} is satisfied by $\shm_0$. Then, if $K\in\mathcal{B}$ is contained in $V_i$ one has:
\begin{enumerate}
\item  $H^j(K; \shm_n)=0, \forall j>0, n\geq 0$;

    \item $H^j(K;\underset{n}{\varprojlim} \shm_n)=0, \forall j>0$.  In particular, if $\shm$ is $\hbar$-complete one also has  $H^j(K; \shm)=0, \forall j>0$.
\end{enumerate}
\end{lemma}
\begin{proof}
(1): Let us consider, for each $n\in\mathbb{N}$, the exact sequence:
$$h^n\shm /h^{n+1} \shm\to \shm_n\xrightarrow{\rho_{0,n}} \shm_{n-1}\to 0.$$
Since $h^n\shm /h^{n+1} \shm$ is the image of the  morphism $\overline{\hbar^n}:\shm_0\to \shm_n$, then it is also a quotient of $\shm_0$. Thus, starting with $\shm_0$, the result  follows by induction on $n$.

(2): By (1), when $\mathcal{B}$ is a basis of open sets  the statement is clear.
When $\mathcal{B}$ is a basis of compact sets, we may consider a fundamental system of compact neighborhoods $\tilde{K}\in \mathcal{B}$ of $K$ in $V_i$. For any $j$, we have $H^j(K, \underset{n}{\varprojlim}\shm_n)\simeq\underset{\tilde{K}}{\varinjlim}\,H^j(\tilde{K}, \underset{n}{\varprojlim}\shm_n).$

Since the map
$H^j(\tilde{K}, \underset{n}{\varprojlim}\shm_n)\to H^j(K, \varprojlim_n\shm_n)$ factors by
$$H^j(\tilde{K}, \underset{n}{\varprojlim}\shm_n)\to H^j(\tilde{K}, \underset{n}{\varprojlim} (\shm_n|_{\tilde{K}})\to H^j(K, \underset{n}{\varprojlim}\shm_n)$$ it remains to observe that  $H^j(\tilde{K}, \underset{n}{\varprojlim}\shm_n|_{\tilde{K}})=0, \text{for}\, j>0,$ as a consequence of $(1)$ and of \cite[Exercise II.12.b)]{KS1}.
\end{proof}

\section{The category $\text{Mod}_{\mathcal{S}}(\mathcal{A})$.}\label{extension-section}

In this section we prepare the notions needed  for our main result (cf. Theorem~\ref{T:1} below). Since we shall deal with  subcategories of sheaves  whose objects  are described by local properties, the convenient language is that of stacks.   Moreover, since on each open subset $U\subset X$ we deal with categories of sheaves  which are  abelian subcategories of modules over some sheaf of rings defined on $X$, and the restriction morphisms are the usual restriction of sheaves to open subsets, our stacks  are in fact sheaves of categories. A fortiori we deal with $\mathbb{K}$-linear stacks. For the background on stacks we refer to \cite{KS3}.

Let $\mathcal{A}$ be an algebra of formal deformation on a complex manifold $X$ and
let there be given and fixed in the sequel a  $\mathbb{K}[[\hbar]]$-linear full Serre substack $\mathcal{S}: U\mapsto \mathcal{S}(U)$ of $\mathscr{M}od_{coh}(\sha)$.
 By convenience, for each ${n\in\N_0}$, we shall   denote by $\mathcal{S}_n$ the  substack of $\mathscr{M}od_{coh}(\mathcal{A}_n)$:
 $$U\mapsto \mathcal{S}_n(U):=\mathcal{S}(U)\cap \text{Mod}(\mathcal{A}_n|_U).$$

Hence, for each open subset $U\subset X$ and each $n\in\N_0$,  $\mathcal{S}_n(U)$ is a full Serre subcategory of $\text{Mod}_{coh}(\mathcal{A}_n|_U)$.

\begin{convention}\label{Conv2} In view of our applications, if there is no ambiguity,  given  an open subset $U\subset X$ and $\shm\in\text{Mod}_{coh}(\mathcal{A}|_U)$,  we shall often use the notation $\shm\in \mathcal{S}$ (resp. $\shm\in \mathcal{S}_n$) to mean that $\shm\in \mathcal{S}(U)$ (resp. $\shm\in \mathcal{S}_n(U)$). Furthermore, we denote by $D^b_{\mathcal{S}}(\sha)$ the full triangulated subcategory of $D^b(\sha)$ consisting of objects with cohomology in $\mathcal{S}$.
\end{convention}
According to the above convention:
\begin{definition}\label{D:12}
We denote by $\text{Mod}_{\mathcal{S}}(\mathcal{A})$ the full subcategory of $\text{Mod}_{coh}(\mathcal{A})$ consisting of $\mathcal{A}$-modules $\shm$ such that:

For each $n\geq 0$, the complex $gr_{\hbar}^n(\shm)$   belongs to $D^b_{\mathcal{S}}(\sha)$, that is,  both $_n\shm$ and $\shm_n$ are objects of $\mathcal{S}_n$.

\end{definition}

Since each  $\shm\in \text{Mod}_{\mathcal{S}}(\sha)$ is coherent, the sequence $(_n\shm)_n$ is locally stationary, in other words $\shm_{\hbar-tor}$ is locally annihilated by a fixed power $\hbar^N$.

\begin{proposition}\label{P:8}
\begin{enumerate}
\item{$\mathcal{S}(X)$ is a subcategory of $\text{Mod}_{\mathcal{S}}(\sha)$.}
\item{Let $\shm$ be an $\hbar$-torsion $\mathcal{A}$-module such that
$\shm\in \text{Mod}_{\mathcal{S}}(\mathcal{A})$. Then
 $\shm\in \mathcal{S}(X).$}
\end{enumerate}
\end{proposition}
\begin{proof}

$(1)$: Let $\shm\in\mathcal{S}(X)$.  For $n\in\N_0$ we have the exact sequences:
$$0\to\hbar^{n+1} \shm\to\shm\to\shm_n\to 0$$ and $$0\to _{n}\shm\to\shm\to\hbar^{n+1}\shm\to 0,$$
thus $_n\shm$ and $\shm_n$ belong to $\mathcal{S}(X)$.

$(2)$: We have $\shm\simeq\shm_{\hbar-tor}$ hence we can cover $X$ by open subsets $U$ and choose  positive integers $N_U$ such that $\hbar^{N_U+1}\shm|_U=0$. Thus $\shm|_U\simeq \shm_{N_U}|_U\in\mathcal{S}_{N_U}(U)\subset \mathcal{S}(U)$ so $\shm\in\mathcal{S}(X)$ since $\mathcal{S}$ is a stack.\end{proof}

\begin{proposition}\label{G9}
Let $\shm$ be a coherent  $\mathcal{A}$-module.
Then the following properties are equivalent:
\begin{enumerate}
\item {$\shm$ is an object of the category $\text{Mod}_\mathcal{S}(\mathcal{A})$;}
\item {$\shm_0\in\mathcal{S}_0$;}
\item {$\shm_n\in\mathcal{S}_n$, for each $n\geq 0$.}
\end{enumerate}
\end{proposition}
\begin{proof}

$(1\Rightarrow 2)$: By definition.

$(2\Rightarrow 3)$: By Lemma \ref{L:19} we have an  exact sequence $$\shm_{n-1}\xrightarrow{\overline{\hbar}}\shm_n\xrightarrow{\rho_{0,n}}\shm_0\to 0.$$
 Since $\shm_0\in\mathcal{S}_0$   we can proceed by induction to
conclude that $\shm_n\in\mathcal{S}_n$ for every $n\geq 0$.

 $(3\Rightarrow 1)$: The statement being of local nature we may assume the existence of $N\geq 0$ such that ${\shm_{\hbar-tor}}_N\simeq\shm_{\hbar-tor}$.

 Assume  that $\shm_n$ belongs to $\mathcal{S}$ for any $n\geq 0$ and let us prove that $_n\shm$  belongs  to $\mathcal{S}$.

 Note that  $_n\shm\simeq _n\shm_{\hbar-tor}$
 and that, by Corollary~\ref{C1},  for each $n\geq 0$, ${\shm_{\hbar-tor}}_n \in \mathcal{S}$.  Taking $N$ big enough as above implies that $\shm_{\hbar-tor}$ belongs to $\mathcal{S}$, so, by Proposition~\ref{P:8}, $\shm_{\hbar-tor}\in
 \text{Mod}_{\mathcal{S}}(\sha)$.  Therefore $_n\shm\in \mathcal{S}$.\end{proof}

\begin{proposition}\label{L:400}
$\text{Mod}_{\mathcal{S}}(\mathcal{A})$ is a Serre subcategory of $\text{Mod}_{coh}(\mathcal{A})$.
\end{proposition}
\begin{proof}
Consider an exact sequence in $\text{Mod}_{coh}(\mathcal{A})$  $$0\to \shm_1\underset{f}{\to}\shm_2\underset{g}{\to}\shm_3\to 0.$$

 One has a distinguished triangle $$gr_\hbar^n(\shm_1)\to gr_\hbar^n(\shm_2)\to gr_\hbar^n(\shm_3)\xrightarrow{+1},$$ so $gr_\hbar^n(\shm_i)\in D^b_{\mathcal{S}}(\mathcal{A})$ if it is so for $gr_\hbar^n(\shm_j)$ and $gr_\hbar^n(\shm_k)$, with $i\neq j, k$, for every $n\in\N_0$.

Therefore it remains to prove that if $\shm_2$ is an object of $\text{Mod}_\mathcal{S}(\mathcal{A})$ then  $\shm_1$ and $\shm_3$ belong to $\text{Mod}_\mathcal{S}(\mathcal{A})$.
 To prove this we consider   the long exact sequence $$0\to_n\shm_1\to_n\shm_2\to_n\shm_3\to\shm_{1,n}\to\shm_{2, n}\to\shm_{3, n}\to 0.$$ The assumption on $\mathcal{S}$ entails that $_n\shm_1, \shm_{3, n}\in\mathcal{S}_n$. By Proposition~\ref{G9}, we also have $_n\shm_3\in\mathcal{S}_n$ and the proof follows.
 \end{proof}

Hence, in view of \eqref{E:111}, an $\mathcal{A}$-module $\shm$ is an object of $\text{Mod}_{\mathcal{S}}(\mathcal{A})$ if and only if $\shm_{\hbar-tor}$ and $\shm_{\hbar-tf}$ are objects of $\text{Mod}_{\mathcal{S}}(\mathcal{A})$.

\section{Extension of  functors.}

Let now $X$, $\sha$ and $\mathcal{S}$ be as in  Section \ref {extension-section}. In the sequel we shall assume that  $\mathcal{S}$ satisfies the following:

\begin{assumption}\label{S}
For each open subset $U\subset X$, $\mathcal{S}(U)=\bigcup_n\mathcal{S}_n(U).$
\end{assumption}
Let $Y$ denote another complex analytic manifold and $\mathcal{A}'$ an algebra of formal deformation on $Y$. For simplicity we still denote by $\hbar$ the fixed section in the center of $\sha'$, thus both $\sha$ and $\sha'$ are $\mathbb{K}[[\hbar]]$-algebras.
Let us denote by $\mathcal{B}'$ the corresponding basis of neighborhoods in $Y$.
For a  substack $\mathcal{S}'$  of  $\mathscr{M}od(\sha')$,  for each $n\in\N_0$, according to the notations of  Section \ref {extension-section}, we shall denote  by $\mathcal{S}'_n$ the substack:
$$\mathcal{S}'_n: V\mapsto \mathcal{S}'_n(V)=\mathcal{S}'(V)\cap\text{Mod}(\mathcal{A}'_n|_V).$$
From now on we consider fixed a full abelian substack $\mathscr{C}'$ of $\mathscr{M}od(\sha')$  as well as a   full Serre substack  $\mathcal{S}'$ of $\mathscr{C}'$.

Firstly assume that we are given a $\mathbb{K}[[\hbar]]$-linear functor $F:\text{Mod}(\sha)\to \text{Mod}(\sha')$. Then, naturally, one  defines a new functor $F^{\hbar}$, formal extension of $F$, by setting:

\begin{definition}\label{G1}
 $F^\hbar$ is the functor  from $\text{Mod}(\mathcal{A})$ to $\text{Mod}(\mathcal{A}')$ :

\begin{enumerate}
 \item{ for $\shm\in \text{Mod}(\mathcal{A}),$
 $$F^\hbar(\shm)=\underset{{n\geq 0}}\varprojlim F(\shm_n),$$}
\item{given a morphism $f:\shm\to\shn$ in $ \text{Mod}(\mathcal{A})$,
  $$F^\hbar(f):F^\hbar(\shm)\to F^{\hbar}(\shn)$$ \noindent is the morphism associated to
the morphisms
\begin{equation*}
 \quad F(\shm_n)\xrightarrow{F(f_n)}F(\shn_n)
\end{equation*}
\noindent where $\shm_n\xrightarrow{f_n}\shn_n$ is induced by $f$.}
\end{enumerate}
\end{definition}

Our goal now is to discuss the properties of $F^{\hbar}$ when $F$ is a functor from $\mathcal{S}$  to $\mathcal{S}'$ (in a sense to be clarified) and  regard its restriction to $\text{Mod}_{\mathcal{S}}(\sha)$. For that we need to state additional assumptions:

\assumption\label{A:v} Henceforward we assume that $\mathcal{S}'$ plays the role of $\mathscr{S}$ in Condition \ref{A} with respect to $\mathscr{C}'$ and $\mathcal{B}'$.

\assumption\label{E:fi}
We fix a functor $\phi$ from the category $Op(X)$ of open subsets of $X$ to the category $Op(Y)$  satisfying the following conditions:

\begin{itemize}
\item{$\Phi(X)=Y$};

\item{For any open subset $\Omega\subset X$ and any open covering $(U_i)_i$ of $\Omega$, $(\Phi(U_i))_i$ is an open covering of $\Phi(\Omega)$.}
\end{itemize}

Let us denote by $\Phi^\ast\mathcal{S}'$ the prestack  defined by assigning
to each open subset $U\subset X$ the subcategory $\Phi^\ast\mathcal{S}'(U):=\mathcal{S}'(\Phi(U))$ of $\text{Mod}(\sha'|_{\Phi(U)})$, the restriction morphism associated to $U\supset V$ being the sheaf restriction from $\Phi(U)$ to $\Phi(V)$.

Let now $F$ be a $\mathbb{K}[[\hbar]]$-linear functor of prestacks: $F:\mathcal{S}\to\Phi^\ast\mathcal{S}'$.
In particular, to each $U\in Op(X)$, $F$ assigns a  $\mathbb{K}[[\hbar]]$-linear functor $F(U):\mathcal{S}(U)\to \mathcal{S}'(\Phi(U))$ compatible with the restriction morphisms in $Op(X)$.

 Whenever there is no ambiguity,   we shall write $F$ instead of $F(X)$.
We shall keep this simplified notation up to the end of this section whenever there is no risk of confusion.

  According to  the preceding conventions, given $\shm\in \mathcal{S}$,  if $\hbar^{n+1}\shm=0$ then $\hbar^{n+1}F(\shm)=0$ hence  $F|_{\mathcal{S}_n}$ takes values in $\mathcal{S}'_n$.

  Let now $U\in Op(X)$ and let us consider $\shm \in\text{Mod}_{\mathcal{S}(U)}(\mathcal{A}|_U)$.

For $n>k\geq 0$ denote by $F(U)(\rho_{k,n})$ the image of the epimorphism $\rho_{k,n}:\shm_n\to\shm_k$ by $F(U)$.

   We obtain  a projective system of $\mathcal{A}'|_{\Phi (U)}$-modules $(F(U)(\shm_n),F(U)(\rho_{k,n}))$ and, by Definition \ref{G1}, an extended functor $F(U)^{\hbar}:\text{Mod}_{\mathcal{S}(U)}(\mathcal{A}|_U) \to \text{Mod}(\mathcal{A}'|_{\Phi(U)})$ given by $F(U)^{\hbar}(\shm):=\underset{{n\geq 0}}\varprojlim F(U)(\shm_n)$.
Recall also that the functor $\varprojlim$ on the category of projective systems of $\text{Mod}(\sha')$ commutes with restriction to open subsets, hence, if we start with $\shm\in\text{Mod}(\sha')$,  for each open subset $U\subset X$, $$F^{\hbar}(\shm)|_{\Phi(U)}\simeq (F|_U)^{\hbar}(\shm|_U).$$

\begin{proposition}\label{G20}
Let $\shm$ be an $\hbar$-torsion $\mathcal{A}$-module in $\text{ Mod}_{\mathcal{S}}(\mathcal{A})$.
Then, we have $F^\hbar(\shm)\simeq F(\shm)$ in $\text{Mod}(\mathcal{A}')$.
\end{proposition}
\begin{proof}
In accordance with Proposition \ref{P:8}, $\shm\in\mathcal{S}$, hence we have a natural morphism $F(\shm)\to F^\hbar(\shm)$. We shall see that this morphism is locally an isomorphism.
We can cover $X$ by open subsets  $U\subset X$ and consider a family of positive integers $N_U$   such that  $\hbar^{N_U+1} \shm|_U=0$.  By the assumption, $\{\Phi(U)\}$ is an open covering of $Y$. Since, for each $n\geq N_U$, $\shm_n|_U\simeq\shm_{N_U}|_U\simeq\shm|_U$, we obtain: $$F^\hbar(\shm)|_{\Phi(U)}:=\varprojlim_{n\geq 0} F(\shm_n)|_{\Phi(U)}
\simeq F(U)(\shm_{N_U}|_U)\simeq F(U)(\shm|_U)\simeq F(\shm)|_{\Phi(U)},$$ which ends the proof.   \end{proof}

As a consequence,  by the assumption on $\mathcal{S}$ we conclude:

$$F^{\hbar}|_{{\mathcal{S}}}\simeq F.$$

\begin{remark}
 The existence of $\Phi$ is the main tool to prove Proposition \ref{G20}, which is a key property in the sequel. $\Phi$ would also be used if, with our machinery in hand, we went on constructing the stack $\mathscr{M}od_{\mathcal{S}}(\mathcal{A})$ defined by $U\mapsto \text{Mod}_{\mathcal{S}}(\mathcal{A})(U)$,
the category $\text{Mod}_{\mathcal{S}}(\mathcal{A})(U)$ being defined in $U$ in a similar way to Definition~\ref{D:12}.
Indeed we might define $F^\hbar$ not only as a morphism of categories but as a functor of prestacks $\mathscr{M}{od}_{\mathcal{S}}(\mathcal{A})\to \Phi^{\ast}\mathscr{M}od(\mathcal{A}')$.
However, in view of the applications, it is enough to work with $F^\hbar$ defined as a morphism of categories, cf. Definition~\ref{G1}.
\end{remark}

\subsection{The case of right exact functors}
In the sequel we will assume that  $F(X)$ is right exact.

\begin{lemma}
Let $\shm\in\text{ Mod}_{\mathcal{S}}(\mathcal{A})$. For each $y\in Y$, if  $K\in\mathcal{B}'$ is contained in a  neighborhood $V$ of $y$ satisfying Condition \ref{A} with respect to $F(\shm_0)$, then one has
 $H^j(K;F(\shm_n))=0$, for any $j>0$ and  $n\in\mathbb{N}_0$.
\end{lemma}

\begin{proof}
In accordance with the right exactness of $F(X)$, for each $n\in\mathbb{N}$, the sequence:  $$F(h^n\shm /h^{n+1} \shm)\to F(\shm_n)\xrightarrow{F(\rho_{n-1,n})} F(\shm_{n-1})\to 0$$ is exact and $F(h^n\shm /h^{n+1} \shm)$ is a quotient of $F(\shm_0)$.

The proof then proceeds by induction  as in Lemma~\ref{L:Dav}(1).
\end{proof}

Let  $\shm\in { Mod}_{\mathcal{S}}(\mathcal{A})$ and let us now denote by $\varrho_n:F^\hbar(\shm)\to F(\shm_{n})$ the canonical projection.

\begin{lemma}\label{L:T1}
 For each $n\geq 1$, the sequence
\begin{equation}\label{exacta3}
F^\hbar(\shm)\xrightarrow{\hbar^{n+1}} F^\hbar(\shm)\xrightarrow{\varrho_n} F(\shm_{n})\to 0.
\end{equation}
is exact.
\end{lemma}
\begin{proof}

Using Lemma \ref{L:19} and considering  sufficiently small  $\Omega$ in a basis $\mathcal{B}'$ in the conditions of Assumption~\ref{AssumptionA} or Assumption~\ref{AssumptionB}, it follows that, for any $N\geq n$, the  sequence
\begin{equation}\label{exacta1}
\Gamma(\Omega;F(\shm_{N}))\xrightarrow{F(\hbar^{n+1})}\Gamma(\Omega; F(\shm_N))\xrightarrow{F(\rho_{n,N})}\Gamma(\Omega;F(\shm_{n}))\to 0
\end{equation}
is exact.  In this way we obtain an exact sequence of projective systems satisfying Mittag-Leffler's condition, so, applying the functor $\underset{N}{\varprojlim}$ we obtain an exact sequence:
\begin{equation}\label{exacta2}
\varprojlim_N\Gamma(\Omega;F (\shm_{N}))\to\varprojlim_N\Gamma(\Omega; F(\shm_N)) \to\Gamma(\Omega;F(\shm_{n}))\to 0.
\end{equation}
 If $\mathcal{B}'$ is a basis of open sets,  this  immediately entails the exactness of (\ref{exacta3}). If $\mathcal{B}'$ is a basis of compact sets, we prove the exactness in the stalks.

 Let $y\in Y$ and let us consider a fundamental system of open neighborhoods $\{\Omega_l\}_{l\in\N}$ of $y$ and a fundamental system of compact neighborhoods $\{K_l\}_{l\in\N}$  of $y$, with $K_l\in\mathcal{B}'$ and
\begin{equation}\label{E:300}
K_{l+1}\subset \Omega_l\subset \text{Int}(K_l).
\end{equation}

Applying  $\underset{l}{\varinjlim}$ to the sequence obtained by replacing in (\ref{exacta2}) $\Omega$ by $K_l$, we obtain an exact sequence: $$F^\hbar(\shm)_y\xrightarrow{\hbar^{n+1}} F^\hbar(\shm)_y\xrightarrow{\varrho_n} F(\shm_{n})_y\to 0,$$ as desired.
\end{proof}
As a consequence,
\begin{corollary}\label{C:T2}
Let $\shm\in { Mod}_{\mathcal{S}}(\mathcal{A})$.
Then $F^\hbar(\shm)$ is an $\hbar$-complete $\mathcal{A}'$-module.
\end{corollary}

As  a consequence of Corollary \ref{C:T2} together with Lemma \ref{L:Dav} we conclude:

\begin{proposition}\label{G18}
Let $\shm\in { Mod}_{\mathcal{S}}(\mathcal{A})$.
Then
\ $F^\hbar(\shm)$ satisfies the vanishing condition (\ref{E:T23}) on Condition \ref{A} for sufficiently small $K\in\mathcal{B}'$.
\end{proposition}

\begin{theorem}\label{G163}
The functor $F^{\hbar}$ is right exact.
\end{theorem}

\begin{proof}
Let  $\shm'\to\shm\to\shm''\to 0$ be an exact sequence in $\text{Mod}_{\mathcal{S}}(\mathcal{A})$.  It gives  an exact sequence of projective families with elements in $\mathcal{S}'$:
\begin{equation}\label{E:227}
F(\shm'_n)\to F(\shm_n)\to F(\shm''_n)\to 0, \,\text{for}\, n\geq 0
\end{equation}

Thus, for every sufficiently small set $\Omega$ in a basis $\mathcal{B}'$  in the conditions of Assumption \ref{AssumptionA} or Assumption~\ref{AssumptionB},
we get  a projective system of exact sequences
\begin{eqnarray}\label{eq5}
\Gamma(\Omega; F(\shm'_n)) \to\Gamma(\Omega; F(\shm_n)) \to \Gamma(\Omega; F(\shm''_n))\to 0,
\end{eqnarray}
where each term satisfies Mittag-Leffler's condition. The proof  then proceeds by the same argument as in Lemma \ref{L:T1}.
\end{proof}

\begin{corollary}\label{C2}
For $\shm\in \text{ Mod}_{\mathcal{S}}(\mathcal{A})$ the sequence below is exact:
\begin{eqnarray}\label{eq10}
F^\hbar(\shm_{\hbar-tor})\to F^\hbar(\shm)\to F^\hbar(\shm_{\hbar-tf})\to 0.
\end{eqnarray}
\end{corollary}

\begin{proposition}\label{T21} Let us assume that $\mathcal{S}'$ is a subcategory of $\text{Mod}_{coh}(\sha')$. Then,
for every $\shm\in\text{Mod}_{\mathcal{S}}(\mathcal{A})$,
$F^\hbar(\shm)$ belongs to $\text{Mod}_{\mathcal{S}'}(\mathcal{A}').$
\end{proposition}

\begin{proof}
In view of Proposition~\ref{G9}  and Lemma \ref{L:T1}
it is enough to prove that $F^\hbar(\shm)$ is $\mathcal{A}'$-coherent, which in turn is reduced to prove  that $$\hbar^n F^\hbar(\shm)/\hbar^{n+1}F^\hbar(\shm)$$ is a coherent $\mathcal{A}_0'$-module by Theorem \ref{T1} together with Corollary \ref{C:T2}.

Since $\hbar^n F^\hbar(\shm)/\hbar^{n+1}F^\hbar(\shm)=\hbar^n F^\hbar(\shm)_n\simeq \hbar^n F(\shm_n)$ the result follows.
\end{proof}

\begin{proposition}\label{P:chc}
Consider the case where each $\mathcal{S}_n$ coincides with the stack $\mathcal{M}\text{od}_{coh}(\sha_n)$.
Assume in addition that $F^\hbar(\sha)$ is $\hbar$-torsion free. Then:
\begin{enumerate}
\item $F^\hbar(\sha)$ is $c\hbar c$.
\item For any $\shm\in\text{Mod}_{coh}(\sha)$, $F^\hbar(\shm)$ is $c\hbar c$.
\end{enumerate}
\end{proposition}
\begin{proof}
Let us start by noticing that, by the assumption, $\text{Mod}_\mathcal{S}(\sha)$ coincides with $\text{Mod}_{coh}(\sha)$.

(1)
The statement follows by Proposition~\ref{P122}, together with Propositions~\ref{G18} and~\ref{C:T2}.

(2) Let us consider a local presentation $$\sha^{N}\to\sha^{L}\to \shm\to 0,$$ for some $N, L\in\N$.
We get an exact sequence $$F^\hbar(\sha)^N\to F^\hbar(\sha)^L\to F^\hbar(\shm)\to 0,$$
and the result follows by Lemma~\ref{chc}.
\end{proof}

\subsection{The case of exact functors}
\textit{Throughout this subsection we shall  assume that $F(U)$ is exact for any open subset $U\subset X$}.

In this case, applying Lemmas~\ref{L:19b} and~\ref{L:T1}, we get a family of exact sequences
$$0\to F^{\hbar}(\shm)_0\overset{\overline{\hbar^n}}{\to} F^{\hbar}(\shm)_n\overset{h}{\to}F^{\hbar}(\shm)_n,\, \forall n.$$
Thus,  again by Lemma~\ref{L:19b}, we conclude:

\begin{corollary}\label{L:301b}
Given $\shm\in\text{Mod}_{\mathcal{S}}(\mathcal{A})$, if $\shm$ is $\hbar$-torsion free then so is $F^{\hbar}(\shm)$.
\end{corollary}

\begin{theorem}\label{T:300}
 $F^\hbar$ is also an exact functor.
\end{theorem}

To prove this we shall need the following results:
\begin{lemma}\label{L:301}
The  sequence of $\mathcal{A}'$-modules
\begin{equation}\label{eq2}
0\to F^\hbar(\shm_{\hbar-tor})\to F^\hbar(\shm)\to F^\hbar(\shm_{\hbar-tf})\to 0
\end{equation}
is exact.
\end{lemma}
\begin{proof}
For each $n\geq 0$, applying the exactness of $F$ to ~\eqref{eqL7}, we obtain  an exact sequence of projective systems with elements in $\mathcal{S}'$:
\begin{equation}\label{E:224}
0\to F((\shm_{\hbar-tor})_n)\to F(\shm_n)\to F((\shm_{\hbar-tf})_n)\to 0.
\end{equation}

Thus, for every sufficiently small set $\Omega$ in a basis $\mathcal{B}'$  in the conditions of Assumption \ref{AssumptionA} or Assumption~\ref{AssumptionB},
we get  a projective system of exact sequences
\begin{equation}\label{E:50}
0\to \Gamma(\Omega; F((\shm_{\hbar-tor})_n)) \to\Gamma(\Omega; F(\shm_n)) \to \Gamma(\Omega;F((\shm_{\hbar-tf})_n))\to 0,
\end{equation}
and the result follows by a similar argument to that used in the proof of Lemma \ref{L:T1}.
\end{proof}

\begin{corollary}\label{C171}
For every $\shm\in{ Mod}_{\mathcal{S}}(\mathcal{A})$ and $n\geq 0$ one has $$_nF^\hbar(\shm)\simeq F(_n\shm).$$
\end{corollary}

\begin{proof}
Fix $n\geq 0$.  By Proposition \ref{G20},
$F^\hbar(\shm_{\hbar-tor})\simeq F(\shm_{\hbar-tor})$ in $\text{Mod}(\mathcal{A}')$.

Then,  Lemma \ref{L:301} and Corollary \ref{L:301b} together with the exactness of $F$ imply the chain of isomorphisms:
$$_nF^{\hbar}(\shm)\simeq _nF^{\hbar}(\shm_{\hbar-tor})\simeq _nF(\shm_{\hbar-tor})\simeq
F(_n\shm_{\hbar-tor})\simeq F(_n\shm).$$
\end{proof}

\begin{lemma}\label{LC}
For any $\shm\in\text{Mod}_{\mathcal{S}}(\mathcal{A})$,  $F^\hbar(\shm)$ is $c\hbar c$.
\end{lemma}
\begin{proof}
By virtue of Corollaries \ref{C:T2} and \ref{L:301b}  and Propositions \ref{G18} and \ref{P122} the assertion holds for
 $\shm$  $\hbar$-torsion free. To treat the general case, we observe that the statement is of local nature on $Y$. We can cover $Y$ by open subsets of the form $\Phi(U)$ and consider integers $N_U$ such that   $\hbar^{N_U}F^{\hbar}(\shm_{\hbar-tor})|_{\Phi(U)}\simeq \hbar^{N_U}F(\shm_{\hbar-tor}|_U)=0$. Since $\sha^{'loc}\simeq \hbar^{N_U}\sha^{'loc}$,  it follows that in $\Phi(U)$ $$R\shh\text{om}_{\sha'}(\sha^{'loc}, F^{\hbar}(\shm_{\hbar-tor}))=0$$ hence  $F^{\hbar}(\shm_{\hbar-tor})$ is $c\hbar c$ and so is $F^{\hbar}(\shm)$ by Lemma~\ref{L:301}.
\end{proof}

As a consequence of Corollary~\ref{C171} together with Lemma~\ref{L:T1} we get:

\begin{corollary}\label{C173}
 For every $\shm\in\text{Mod}_{\mathcal{S}}(\mathcal{A})$ and $n\geq 0$, we have a family of isomorphisms in $\text{Mod}(\mathcal{A}'_n):$
\begin{equation}\label{E:12}
H^j(gr_\hbar^n(F^\hbar(\shm)))\simeq F(H^jgr_\hbar^n(\shm)), \, \forall j\in\Z.\end{equation}
\end{corollary}

\textit{End of the proof of Theorem \ref{T:300}}

Given an exact sequence in $\text{Mod}_{\mathcal{S}}(\mathcal{A})$
\begin{equation}\label{E:304}
0\to\shm'\to\shm\to\shm''\to0
\end{equation}
we deduce that $$0\to F^{\hbar}(\shm')\to F^{\hbar}(\shm)\to F^{\hbar}(\shm'')\to0$$ is exact thanks to Lemma \ref{LC} and Corollary \ref{C173} by applying $gr_{\hbar}$ to (\ref{E:304}).
This achieves the proof of Theorem \ref{T:300}.

\hspace{2mm}

\subsection{Unicity of extensions}

Let us now discuss the unicity of the extensions of the functors treated above.

Consider  Serre substacks  $\mathcal{S}$ and $\mathcal{S}'$ and a functor $\Phi:Op(X)\to Op(Y)$, as above.
Let $G :\mathcal{S}\to\Phi^{\ast}\mathcal{S}'$ be a $\mathbb{K}[[\hbar]]$-linear functor.
Let $\mathcal{G}$  be a functor  from $\text{Mod}_{\mathcal{S}}(\mathcal{A})$ to $\text{Mod} (\mathcal{A}')$ such that  $\mathcal{G}|_{\mathcal{S}}$ takes values in $\mathcal{S}'$.

\begin{definition}\label{D:T}
We shall say that \textit{$\mathcal{G}$ extends $G(X)$} if $\mathcal{G}|_{\mathcal{S}}$ and $G$ are isomorphic functors.\end{definition}
In particular, if $\mathcal{G}$ extends $G(X)$, the natural morphisms
 $$\mathcal{G}(\shm)\to \mathcal{G}(\shm_n)\simeq G(\shm_n)$$ define a morphism of functors:$$\mathcal{G}(\cdot)\to G^{\hbar}(\cdot).$$

\begin{proposition}\label{P:unicidade}
Consider the case where each $\mathcal{S}_n$ coincides with the stack $\mathcal{M}\text{od}_{coh}(\sha_n)$.

Assume that $G(X)$ is right exact.
Then, up to isomorphism, $G^{\hbar}$ is the unique right exact functor
$\mathcal{G}:\text{Mod}_{coh}(\sha)\to\text{Mod}(\sha')$
that extends $G(X)$ and verifies $\mathcal{G}(\sha)= \underset{n}{\varprojlim} G(\sha_n)$.
\end{proposition}
\begin{proof}
Recall that  $\text{Mod}_\mathcal{S}(\sha)$ coincides with $\text{Mod}_{coh}(\sha)$.
First of all,  it is clear that $G^\hbar$ satisfies the statement.

Suppose that $\mathcal{G}$ is another right exact functor that extends $G(X)$.
Taking a local presentation of $\shm\in\text{Mod}_{coh}(\sha)$, say,
$$\sha^{N}\to\sha^{L}\to \shm\to 0,$$ and applying $\mathcal{G}$ and $G^\hbar$,
one gets the diagram below with exact rows:
$$\begin{matrix}\mathcal{G}(\sha)^N&\to&\mathcal{G}(\sha)^L&\to& \mathcal{G}(\shm)&\to& 0&\to& 0\\
\downarrow&&\downarrow&&\downarrow&&\downarrow&&\downarrow\\
(G^\hbar(\sha))^N&\to&(G^\hbar(\sha))^L&\to& G^\hbar(\shm)&\to& 0&\to& 0.
\end{matrix}$$
The statement then follows by the Five Lemma in view of the hypothesis $\mathcal{G}(\sha)=G^\hbar(\sha)$.
\end{proof}

\begin{proposition}\label{P:unicidade2}
Consider the case where $G(U):\mathcal{S}(U)\to \mathcal{S}'(U)$ is an exact functor for any $U\in Op(X)$.
Then, up to isomorphism, $G^{\hbar}$ is the  unique (exact) functor $\mathcal{G}$ that extends $G(X)$, takes values in the category of $c\hbar c$ objects and verifies $_n\mathcal{G}(\shm)\simeq\mathcal{G}(_n\shm)$ and $\mathcal{G}(\shm)_n\simeq \mathcal{G}(\shm_n)$ (the last isomorphisms being associated to the canonical morphisms). %
\end{proposition}
\begin{proof}
Clearly, $G^\hbar$ satisfies the statement.

On the other hand, consider a
right exact functor $\mathcal{G}$
which extends $G$, goes to the category of $c\hbar c$ $\sha'$-modules and commutes with $_n(\cdot)$ and $(\cdot)_n$.
Then, applying $gr$ to the morphism $\mathcal{G}\to G^\hbar$, one concludes the isomorphism $\mathcal{G}\simeq G^\hbar$.
\end{proof}

We can now sum up the above discussion and state the main result of this section:

\begin{theorem}\label{T:1} Let $X$ and $Y$ be complex manifolds, let $\mathcal{A}$ (resp. $\mathcal{A}'$) be an algebra of formal deformation on $X$ (resp. on $Y$), let $\mathcal{S}$ (resp. $\mathcal{S}'$) be a  full Serre substack of $\mathscr{M}od_{coh}(\mathcal{A})$  (resp.  a full Serre substack of a full   substack $\mathscr{C}'$ of abelian categories of $\mathscr{M}od(\mathcal{A}'))$  and let be given a  functor $\Phi:Op(X)\to Op(Y)$ in the conditions of \ref{E:fi}. Assume that $\mathcal{S}$ satisfies assumption~\ref{S} and that $\mathcal{S}'$ satisfies  assumption~\ref{A:v} with respect to $\mathscr{C}'$.
Let $F:\mathcal{S}\to\Phi^\ast\mathcal{S}'$ be a $\mathbb{K}[[\hbar]]$-linear functor and assume that for each open subset $U$, $F(U)$ is right exact. Then:
\begin{enumerate}
\item { $F^{\hbar}:\text{Mod}_{\mathcal{S}}(\mathcal{A})\to \text{Mod}(\mathcal{A}')$ is a canonical right exact $\mathbb{K}[[\hbar]]$-linear  extension of $F$;}
\item{when $\mathscr{C}'=\mathscr{M}od_{coh}(\mathcal{A}')$, then $F^\hbar$ takes values in $\text{Mod}_{\mathcal{S}'}(\mathcal{A}')$;}

\item{if, for each open subset $U\subset X$,   $F(U)$ is exact,  then so is $F^{\hbar}$, and up to isomorphism, it is the unique  extension of $F$   that takes values in the category of $c\hbar c$ objects and commutes with $_n(\cdot)$ and $(\cdot)_n$.}
\end{enumerate}
\end{theorem}

\remark\label{P:nova} $F^{\hbar}$ is canonical in the following sense: keeping the preceding notations,
if we  are given  a functor $H: \mathcal{S}\to\tilde{\mathcal{S}}$, a functor  $\tilde{H}: S'\to\tilde{S}'$, and a  morphism $\theta$ of functors  $\Phi, \Psi:Op(X)\to Op(Y)$, we derive  a functor of prestacks $\tilde{H}^*:\Phi^{\star}\mathcal{S}'\to \Psi^{\star}\tilde{\mathcal{S}'}$, together with an extension $H^{\hbar}:\text{Mod}_{\mathcal{S}}(\sha)\to\text{Mod}_{\tilde{\mathcal{S}}}(\sha')$.
If moreover $F: \mathcal{S}\to\Phi^*{S}'$, $\tilde{F}:\tilde{\mathcal{S}}\to\Psi^*\tilde{\mathcal{S}'}$ are given,   one may  define a morphism $F\to\tilde{F}$ as  being a morphism of functors $\tilde{H}^*\circ F\to \tilde{F}\circ H$ (cf. diagram below):

$$\begin{matrix}\mathcal{S}& \xrightarrow{F} & \Phi^{\star}\mathcal{S}'\\ H\downarrow & & \downarrow \tilde{H}^*\\ \tilde{ \mathcal{S}}&\xrightarrow{\tilde {F}} &\Psi^{\star}\tilde{\mathcal{S}'} \end{matrix}$$
In this situation, we get a morphism of functors $F^{\hbar}\to\tilde{F}^{\hbar}\circ H^{\hbar}$ (cf. diagram below):
$$\begin{matrix}\mathcal{S}& \hookrightarrow &\text{Mod}_{\mathcal{S}}(\sha)&\xrightarrow{F^\hbar}&\text{Mod}(\sha') \\ H\downarrow & & \downarrow H^{\hbar}& &||\\ \tilde{ \mathcal{S}}&\hookrightarrow &\text{Mod}_{\tilde{\mathcal{S}}}(\sha)&\xrightarrow{\tilde{F}^\hbar}&\text{Mod}(\sha') \end{matrix}$$



\section{Application to $\shd_X[[\hbar]]$-modules.}\label{S2}

Let $X$ be  a complex manifold.
Let $\sho_X$ be the sheaf of holomorphic functions on $X$ and let
$\shd_X$ be the sheaf over $X$ of linear holomorphic differential operators of finite order.

 As the title suggests, in this section we apply the results of Section \ref{extension-section} and of Section 3  for $\mathcal{A}=\shd_X[[\hbar]]$.  We shall extend  functors  defined    on full  Serre subcategories of $\text{Mod}_{coh}(\shd_X)$ whose objects  are characterized by local properties.  As we shall see, these full subcategories being the data of full Serre substacks,  the functors  we are interested in define linear functors  to which apply the results in Section 3.
 So we skip the constant reference to substacks, as stated in Convention \ref{Conv2}, referring to the categories most of the time.

Recall that one  denotes $\shd_X[[\hbar]]$ by $\shd_X^{\hbar}$ as well as $\sho_X[[\hbar]]$ by $\sho_X^{\hbar}$.  Recall also  that $\shd_X^{\hbar}$ satisfies (i), (ii) and (iv) of Assumption~\ref{AssumptionB} taking for $\mathcal{B}$ the family of Stein compact subsets of $X$, for $\mathcal{A}_0$ the $\C$-algebra $\shd_X$ and considering the prestack of good $\shd_X$-modules in the sense of~\cite{Ka2}.

The formal extension functor is defined by
\begin{eqnarray*}
(\cdot)^{\hbar}:\text{Mod}(\shd_X)&\to&\text{Mod}(\shd_X^{\hbar}),  \\
\shm&\mapsto& \shm^\hbar=\varprojlim_{n\geq 0}(\shd_X^{\hbar}/\hbar^{n+1}\shd_X^{\hbar}\otimes_{\shd_X}\shm).
\end{eqnarray*}

In particular $\shm^{\hbar}$ is $\hbar$-complete for any $\shm\in\text{Mod}(\shd_X)$.

An exhaustive study of $\shd_X^\hbar$-modules has been done in~\cite{DGS}  whose notations we  maintain here.

\begin{remark}\label{R:St}
$\shd^{\hbar}_X$ induces,  for each $n\geq 0$, a (left and right) structure of free $\shd_X$-module  of finite rank ($n+1$) on  the algebra $\shd^{\hbar}_{X,n} :=\shd_X^{\hbar}/\hbar^{n+1}\shd_X^{\hbar}$ which becomes a $(\shd_X^{\hbar},\shd_X)$-bimodule (resp. a $(\shd_X,\shd_X^{\hbar})$-bimodule).
\end{remark}

Following \cite{KS2},
an object $\shm\in D^b_{coh}(\shd_X^\hbar)$ is said to be holonomic
(resp. regular holonomic)
if $gr_\hbar(\shm)$ is an object of $D^b_{hol}(\shd_X)$
(resp. of $D^b_{hol}(\shd_X)$).
The full subcategory of $D^b_{coh}(\shd_X^\hbar)$ of holonomic (resp. regular holonomic)
objects is denoted by $D^b_{hol}(\shd_X^\hbar)$ (resp. $D^b_{rh}(\shd_X^\hbar)$).

Denote by $\Omega_X$ the sheaf of holomorphic forms of maximal degree on $X$
and set $\Omega_X^{\otimes -1}:=\shh\text{om}_{\sho_X}(\Omega_X,\sho_X)$ as usual.

We shall need the following functors:
\begin{eqnarray*}
&& D^{\prime}_{\C^\hbar}:D^b(\C^{\hbar}_X) \to D^b(\C^{\hbar}_X),
\quad F \mapsto R\shh\text{om}_{\C_X^{\hbar}}(F, \C^{\hbar}_X),\\
&& D^{\prime}_{\shd^\hbar}:D^b(\shd^{\hbar}_X) \to D^b(\shd^{\hbar}_X),
\quad \shm \mapsto R\shh\text{om}_{\shd_X^{\hbar}}(\shm, \shd^{\hbar}_X),\\
&&Sol_{\hbar}:D^b_{coh}(\shd_X^\hbar) \to D^b(\C^{\hbar}_X),
\quad \shm \mapsto R\shh\text{om}_{\shd_X^\hbar}(\shm,\sho_X^\hbar),\\
&& DR_{\hbar}:D^b_{coh}(\shd_X^\hbar)\to D^b(\C^{\hbar}_X),
\quad \shm\mapsto R\shh\text{om}_{\shd_X^{\hbar}}(\sho_X^{\hbar},\shm),\\
&& \mathbb{D}_{\hbar}:D^b(\shd_X^\hbar)  \to D^b(\shd_X^\hbar),
\quad \shm\mapsto R\shh\text{om}_{\shd_X^\hbar}(\shm,\shd_X^\hbar\otimes_{\sho_X}\Omega_X^{\otimes-1})[d_X].
\end{eqnarray*}

Note that both $D^{\prime}_{\shd^\hbar}$ and $\mathbb{D}_{\hbar}$ preserve coherence.

When the base ring is fixed and there is no risk of confusion we shall denote each  of the functors $D^{\prime}_{\C^\hbar}$
and  $D^{\prime}_{\shd^\hbar}$ simply by $D^{\prime}_\hbar$.

As shown in Theorem~3.15 of~\cite{DGS}  the following diagram is commutative:
\begin{eqnarray}\label{eq11}
&&\xymatrix{
D^b_{hol}(\shd_X^\hbar)\ar[d]^-{Sol_{\hbar}} \ar[r]^-{DR_\hbar} & D^b_{\C-c}(\C^\hbar_X) \ar[dl]^-{D^\prime_{\hbar}}\\
D^b_{\C-c}(\C^\hbar_X).
}\end{eqnarray}

\begin{remark}\label{stack}
After \cite {DGS}, and according with our previous notations, the category   $\text{Mod}_{coh}(\shd^{\hbar}_X)$ equals the category $\text{Mod}_{\mathcal{S}}(\shd_X^{\hbar})$, where the full Serre substack $\mathcal{S}$  is $$U\mapsto\mathcal{S}(U)=\cup_n \mathscr{M}od_{coh}({\shd_X^{\hbar}}_{,n})(U).$$ Similarly, the category $\text{Mod}_{rh}(\shd^{\hbar}_X)$ is defined by the Serre substack $$U\mapsto\mathcal{S}(U)=\cup_n\mathscr{M}od_{rh}({\shd_X^{\hbar}}_{,n})(U).$$
\end{remark}

\subsection{Inverse image.}
Let $f: Y\to X$ be a morphism of complex manifolds.   One defines a right exact functor $\underline{f}^{\ast}: \bigcup_n \text{Mod}(\shd^{\hbar}_{X,n})\to\bigcup _n\text{Mod}(\shd_{Y,n}^{\hbar})$ setting $$\underline{f}^{\ast}(\shm)=\sho_Y\otimes_{f^{-1}\sho_X}f^{-1}\shm.$$ We refer, among others, to \cite{LS} for a quite general study of this functor for $n=0$.

Let $\mathscr{C}'$ be the abelian full subcategory of pseudocoherent $\shd_Y^{\hbar}$-modules and let $\mathcal{S}'\subset \mathscr{C}'$  be the full subcategory of pseudocoherent $\shd_Y^{\hbar}$-modules satisfying Assumption ~\ref{A} with respect to the basis $\mathcal{B}'$ of Stein compact subsets of $Y$.

\begin{lemma}\label{L:111}
Let $\mathcal{S}$ be equal to $\bigcup_n \text{Mod}_{coh}({\shd_{X,n}^{\hbar}})$. Then,  for any morphism $f:Y\to X$ and any $\shm\in\mathcal{S}$, $\underline{f}^{\ast}(\shm)\in\mathcal{S}'$.
\end{lemma}

\begin{proof}
Observe that, for given $n\geq 0$, and $\shm\in\mathcal{S}_n$, considering the $f^{-1}(\shd_X)$-module structure on $f^{-1}(\shm)$ referred to in Remark \ref{R:St}, we get
$$\underline{f}^{\ast}(\shm)\simeq \shd_{Y\to X}\otimes_{f^{-1}\shd_{X}} f^{-1}\shm.$$
 Recall that any coherent $\shd_Y$-module is locally good, and any pseudocoherent $\shd_Y$-submodule of a good $\shd_Y$-module is itself good.

By \cite{LS}, it is known that the inverse image of a coherent $\shd_X$-module $\shm$ is a pseudocoherent $\shd_Y$-module which satisfies the following property:
\begin{enumerate}
\item{In a suitable neighborhood of each $y\in Y$, it is an inductive limit of good $\shd_Y$-submodules.}
\end{enumerate}

Since inductive limits commute with cohomology on  compact sets, it follows that $\underline{f}^{\ast}(\shm)$ satisfies (\ref{E:T23}).
Note also that condition $(1)$ is closed for quotients and hence for submodules in the abelian category of pseudocoherent modules.
Indeed, given $\tilde{\shm}$ a pseudocoherent module satisfying $(1)$ and given a pseudocoherent submodule $\tilde{\shn}$ of  $\tilde{\shm}$, the quotient $$\tilde{\shm}/\tilde{\shn}$$ is pseudocoherent. If in an open set $\Omega$ we have $\tilde{\shm}|_{\Omega}\simeq \underset{\alpha}{\varinjlim} \shm_{\alpha}$, for given good submodules $\shm_{\alpha}$ of $\tilde{\shm}$, since their images in $\tilde{\shm}/\tilde{\shn}$ are locally finitely generated, hence coherent, hence good, it follows that each $\tilde{\shn}\cap\shm_{\alpha}$ is good. By the exactness of inductive limits we get that $\tilde{\shn}|_{\Omega}\simeq \underset{\alpha}{\varinjlim} \tilde{\shn}\cap\shm_{\alpha}$.
This ends the proof.
\end{proof}

In what follows we shall denote by $\mathcal{S}$ the full Serre substack $$U\mapsto \mathcal{S}(U)=\bigcup_n \mathscr{M}od_{coh}(\shd_{X,n})(U)$$ of $\mathscr{M}od_{coh}(\shd_X^{\hbar})$.

Let us denote by $\Phi:Op(X)\to Op(Y)$ the functor given by $\Phi(U)=f^{-1}(U)$ together with the inclusions $U\supset V\mapsto \Phi(U)\supset\Phi (V)$. Clearly $\Phi$ satisfies \ref{E:fi}.

In view of Remark \ref{stack} and Convention \ref{Conv2}, by Theorem \ref{T:1}
 we are in the conditions to define  a right exact functor extending $\underline{f}^{\ast}$:
$$\underline{f}^{\ast,\hbar}:\text{Mod}_{coh}(\shd^{\hbar}_X)\to \text{Mod}(\shd^{\hbar}_Y),$$
given by $$\underline{f}^{\ast,\hbar}(\shm)=\underset{n\geq 0}{\varprojlim} \left(\sho_Y\underset{f^{-1}(\sho_X)}{\otimes}f^{-1}\shm_n\right).$$
and we have:
\begin{equation}\label{E1}
\underline{f}^{\ast,\hbar}(\sho_X^\hbar)\simeq\sho_Y^\hbar.
\end{equation}
Indeed, one has
\begin{eqnarray*}
\underline{f}^{\ast,\hbar}(\sho_X^\hbar)&=&\varprojlim_{n\geq 0}\left(\sho_Y\underset{f^{-1}(\sho_X)}{\otimes}f^{-1}(\sho_X^\hbar/\hbar^{n+1}\sho_X^\hbar)\right) \\
&\simeq& \varprojlim_{n\geq 0} \left(\sho_Y\underset{f^{-1}(\sho_X)}{\otimes}f^{-1}(\sho_X\underset{\C_X}{\otimes}(\C_X^\hbar/\hbar^{n+1}\C_X^\hbar))\right)\\
&\simeq& \varprojlim_{n\geq 0}\left(\sho_Y\underset{\C_Y}{\otimes}(\C_Y^\hbar/\hbar^{n+1}\C_Y^\hbar)\right)\simeq\sho_Y^\hbar.
\end{eqnarray*}

Let us consider the $(\shd_Y^\hbar, f^{-1}(\shd_X^\hbar))$-bimodule $$\mathcal{K}:=\underline{f}^{\ast,\hbar}(\shd_X^\hbar)=\varprojlim_{n\geq 0} \left(\sho_Y\underset{f^{-1}(\sho_X)}{\otimes}f^{-1}(\shd_X^\hbar/\hbar^{n+1}\shd_X^\hbar)\right).$$
Since for each $n$, $f^{-1}(\shd_X^\hbar/\hbar^{n+1}\shd_X^\hbar)$ is isomorphic to $f^{-1}(\shd_X)\otimes_{\C_Y} \C_Y^\hbar/\hbar^{n+1}\C_Y^\hbar$ we conclude:
\begin{lemma}\label{L:403}
As a $(\shd_Y^\hbar, f^{-1}(\shd_X^\hbar))$-bimodule $\mathcal{K}$ is  isomorphic to the formal extension $(\shd_{Y\to X})^{\hbar}$ of the transfer module $\shd_{Y\to X}$. In particular it is $\hbar$-complete.
\end{lemma}

\begin{proposition}\label{KK}
Let $f:Y\to X$ be a morphism. Then:
\begin{enumerate}[{\rm (i)}]

\item $\mathcal{K}$ is  $c \hbar c$.

\item For each $\shm\in\text{Mod}_{coh}(\shd_X^{\hbar})$,  $\underline{f}^{\ast,\hbar}(\shm)$ is $c \hbar c$.

\item For $\shm\in \text{Mod}_{coh}(\shd_X^{\hbar})$, one has an isomorphism in $\text{Mod}(\shd_Y^{\hbar})$:
\begin{equation}\label{K}
\mathcal{K}\underset{f^{-1}(\shd^{\hbar}_X)}{\otimes}f^{-1}\shm\simeq \underline{f}^{\ast,\hbar}(\shm).
\end{equation}
\item For each $\shm\in\text{D}^b_{coh}(\shd_X^{\hbar})$, $\mathcal{K}\overset{L}{\underset{f^{-1}(\shd^{\hbar}_X)}{\otimes}}f^{-1}\shm$ is  $c \hbar c$.

\end{enumerate}

\end{proposition}
\begin{proof}
(i) Follows by Proposition~\ref{P:chc}(1) since $\mathcal{K}\simeq\shd_{Y \to X}^\hbar$ is $\hbar$-torsion free.

(ii) Follows by Proposition~\ref{P:chc}(2).

To prove (iii), note that $\shm\mapsto\mathcal{K}\underset{f^{-1}(\shd_X^\hbar)}{\otimes}f^{-1}(\shm)$ is a right exact functor
that extends $\underline{f}^\ast$ in the sense of Definition~\ref{D:T}. Hence, the result follows by Proposition~\ref{P:unicidade}.

(iv) Let now be given $\shm\in \text{D}^b_{coh}(\shd_X^{\hbar})$. Given a local free resolution  $$\shd_X^{\hbar,\bullet}\overset{QIS}{\to}\shm$$ it yields a quasi-isomorphism $$\mathcal{K}^{\bullet}\overset{QIS}{\to}\mathcal{K}\overset{L}{\underset{f^{-1}(\shd^{\hbar}_X)}{\otimes}}f^{-1}\shm$$ in $D^b(\shd_Y^{\hbar})$. To conclude the statement it is enough to apply  Lemma \ref{chc}.
\end{proof}

\begin{remark}\label{Der}As a consequence of (iii)  of Proposition \ref{KK}, we give a meaning to
  $L\underline{f}^{\ast,\hbar}$ as follows:

 For $\shm\in D^b_{coh}(\shd_X^{\hbar})$ we set: $$L\underline{f}^{\ast,\hbar}(\shm):=\mathcal{K}\overset{L}{{\underset{f^{-1}(\shd^{\hbar}_X)}{\otimes}}}f^{-1}\shm.$$

More precisely, the left hand side of (\ref{K})  defines a left derivable right exact functor  $I_f$ on $\text {Mod}(\shd_X^{\hbar})$ which is equivalent to  $\underline{f}^{\ast,\hbar}$ on $\text {Mod}_{coh}(\shd_X^{\hbar})$. Since  any $\shm\in\text {Mod}_{coh}(\shd_X^{\hbar})$ admits locally a free, hence $I_f$-projective, resolution, we may denote without ambiguity the derived functor  $\mathcal{K}\overset{L}{{\underset{f^{-1}(\shd^{\hbar}_X)}{\otimes}}}f^{-1}(\cdot)$ by  $L\underline{f}^{\ast,\hbar}(\cdot)$.
\end{remark}

\begin{proposition}\label{P:hol}
Let $\shm\in\text{D}^b_{hol}(\shd_X^\hbar)$.
Then $L\underline{f}^{\ast,\hbar}(\shm)\in \text{D}^b_{hol}(\shd_Y^\hbar)$. The same statement holds if we replace the assumption of holonomicity by that of regular holonomicity.
\end{proposition}
\begin{proof}
Since $gr_\hbar(L\underline{f}^{\ast,\hbar}(\shm))\simeq  \shd_{Y\to X}\overset{L}{{\underset{f^{-1}(\shd_X)}{\otimes}}}f^{-1}gr_\hbar(\shm)$ (indeed as shown in Proposition 1.4.3 of \cite{KS2}, $gr_\hbar$ commutes with tensor product and also with $f^{-1}$) the result follows from the analogous property for holonomic $\shd$-modules due  to Kashiwara  (\cite{Ka0}) together with Proposition \ref{P121}.
\end{proof}

\subsection{The non characteristic inverse image}

Recall that, in the sense of \cite{KS1},  $f$ is said to be non-characteristic  for  $\shm\in \text{Mod}_{coh}(\shd_X)$  if $$f_\pi^{-1}(Char(\shm))\cap \ker f_d\subset Y\times_ X T_X^*X,$$ where $Char(\shm)$ denotes the characteristic variety of $\shm$.

Let us now denote by $NC(f)$ the Serre substack  of $\mathscr{M}od_{coh}(\shd_X)$  which, to each open subset $U\subset X$,  assigns $NC(f)(U)$, the full Serre subcategory whose objects $\shm\in \text{Mod}_{coh}(\shd_X|_U)$ are such that $f|_{f^{-1}(U)}$ is non-characteristic for $\shm$.

We can restrict $\underline{f}^{\ast}$ to  $NC(f)$ as a $\mathbb{\C}^{\hbar}$-linear functor of stacks. Then, for each open subset $U\subset X$,  $\underline{f}^{\ast}(U)$ is exact (\cite{KS1}, Proposition 11.2.12),  and takes values in $\text{Mod}_{coh}(\shd_Y|_{f^{-1}(U)})$.

Therefore, by Theorem \ref{T:1}, the restriction of the extension functor $\underline{f}^{\ast,\hbar}$ to $\text{Mod}_{NC(f)}(\shd^{\hbar}_X)$ is an exact functor $$\underline{f}^{\ast,\hbar}: \text{Mod}_{NC(f)}(\shd^{\hbar}_X)\to \text{Mod}_{coh}(\shd^{\hbar}_Y).$$
We shall denote by $D^b_{NC(f)}(\shd_X^{\hbar})$ the subcategory of $D^b_{coh}(\shd_X^{\hbar})$ whose objects $\shm$ are such that $gr_{\hbar}(\shm)$ is non-characteristic  for $f$.

In particular, for any $f$, $\sho_X^\hbar\in \text{Mod}_{NC(f)}(\shd^{\hbar}_X)$.

Recall that for any coherent $\shd_X$-module one has a well defined morphism in $D^b(\shd_Y)$:
\begin{equation}\label{E:f}
f^{-1}(R\shh \text{om}_{\shd_X}(\shm, \sho_X))\to R\shh \text{om}_{\shd_Y}(L\underline{f}^{\ast}(\shm), \sho_Y),
\end{equation}
which is an isomorphism when $\shm$ is non-characteristic for $f$ (Cauchy-Kowalewskaia-Kashiwara's Theorem).

This result may be generalized to the formal setting as follows:

\begin{theorem}[Cauchy-Kowalewskaia-Kashiwara]\label{CKKT}
Assume that $\shm$ belongs to $D^b_{NC(f)}(\shd^{\hbar}_X)$. Then one has a natural isomorphism in
$\text{D}^b(\C_Y^{\hbar})$:
\begin{equation}\label{E:f3}
f^{-1}R\shh \text{om}_{\shd_X^{\hbar}}(\shm, \sho_X^{\hbar})\simeq  R\shh \text{om}_{\shd_Y^{\hbar}}(L\underline{f}^{\ast,\hbar}(\shm), \sho_Y^{\hbar}).
\end{equation}
\end{theorem}
\begin{proof}

 By Propositions \ref{P1} and \ref{P:500} we have a natural morphism between $c \hbar c$ objects $$f^{-1}R\shh \text{om}_{\shd_X^{\hbar}}(\shm, \sho_X^{\hbar})\to  R\shh \text{om}_{\shd_Y^{\hbar}}( \mathcal{K}\overset{L}{\underset{f^{-1}(\shd_X^\hbar)}{\otimes}}f^{-1}(\shm), \mathcal{K}\overset{L}{\underset{f^{-1}(\shd_X^\hbar)}{\otimes}}f^{-1}\sho_X^{\hbar})$$ (see Exercise II.24 of~\cite{KS1} for the construction of the morphism).

Besides, by (\ref{E1}) and (\ref{K}),  $\mathcal{K}\underset{f^{-1}(\shd_X^\hbar)}{\otimes}f^{-1}(\sho_X^{\hbar})\simeq\sho_Y^\hbar$.
The result then follows by Proposition \ref{KK} and Proposition \ref{P121}.
\end{proof}
We may also introduce
the so called \textit{extraordinary inverse image} associated to $f$, which we
denote by $L\underline{f}^{!,\hbar}$:
\begin{eqnarray*}
L\underline{f}^{!,\hbar}:D^b_{coh}(\shd_X^\hbar)\to D^b(\shd_Y^\hbar),
\quad \shm\mapsto \mathbb{D}_\hbar(L\underline{f}^{\ast,\hbar}(\mathbb{D}_\hbar(\shm))).
\end{eqnarray*}
We refer to~\cite{Me} for that notion in the  $\shd$-module case.

By Proposition~\ref{P:hol} and by duality
 we get:

\begin{corollary}\label{C:hol}
Let $\shm\in D^b_{hol}(\shd_X^\hbar)$. Then, $L\underline{f}^{!,\hbar}(\shm)\in D^b_{hol}(\shd_Y^\hbar)$.
\end{corollary}

\subsection{Direct image.}
In this section we discuss the possible application or adaptation of our results to the functor of direct image. We shall work with right $\shd^\hbar$-modules but all the results  are easily adapted to the case of left $\shd^\hbar$-modules.

We identify the abelian category of right $\shd^\hbar$-modules with the category $\text{Mod}({\shd_Y^\hbar}^{op})$.

Let $f:Y\to X$ be a morphism of complex manifolds. Let $\mathcal{K}$ denote the associated transfer module.

\subsubsection{The case of a closed embedding}

 Let us treat the case  where $f=i:Y\hookrightarrow X$, the embedding of a closed submanifold.
In this case $\shd_{Y\to X}$ is flat over $\shd_Y$  and we obtain an exact functor
$$\underline{i}_{\ast}:\text{Mod}_{coh}({\shd_Y}^{op})\to \text{Mod}_{coh}({\shd_X}^{op}), \quad \shm \mapsto \underline{i}_{\ast}(\shm):= i_\ast(\shm\otimes_{\shd_Y} \shd_{Y\to X}).$$
Here the full Serre substacks $\mathcal{S}$ and $\mathcal{S}'$ are respectively $\mathscr{M}od_{coh}({\shd_Y}^{op})$ and  $\mathscr{M}od_{coh}({\shd_X}^{op})$.
We can choose as a candidate for the functor $\Phi:Op(Y)\to Op(X)$ the data $U\mapsto \Phi(U):=X\setminus(Y\setminus U)$ which clearly satisfies \ref{E:fi} and we are in conditions to apply Theorem \ref{T:1} to
 extend $\underline{i}_{\ast}$ as an exact functor
$$\underline{i}_{\ast}^\hbar: \text{Mod}_{coh}({\shd_Y^\hbar}^{op})\to\text{Mod}_{coh}({\shd_X^\hbar}^{op}),$$  $$\underline{i}_{\ast}^{\hbar}(\shm):=\varprojlim_{n\geq 0} i_\ast(\shm_n\otimes_{\shd_Y}\shd_{Y\to X}).$$
\subsubsection{Discussion of the general case}

By Lemma \ref{L:403}  we have $\mathcal{K}\simeq (\shd_{Y\to X})^{\hbar}$, hence $$\mathcal{K}_n\simeq \shd_{Y,n}^\hbar\otimes_{\shd_Y}\shd_{Y\to X}.$$  So,  for $\shm\in\text{Mod}({\shd_Y^{\hbar }}^{op})$,
we get natural isomorphisms in $\text{Mod}(f^{-1}(\shd_X)^{\hbar}_n)$:
 \begin{equation}\label{E:T21}
 \shm_n\otimes_{\shd_Y^\hbar} \mathcal{K}\simeq \shm_n\otimes_{\shd_{Y}^\hbar}\mathcal{K}_n
 \simeq (\shm_n\otimes_{\shd_Y^{\hbar}}\shd_{Y,n}^{\hbar})\otimes_{\shd_Y}\shd_{Y\to X}
 \simeq \shm_n\otimes_{\shd_Y} \shd_{Y\to X}.
 \end{equation}

Since projective limits commute with direct images, (\ref{E:T21}) entails a morphism
\begin{equation}\label{eq:405}
f_\ast(\shm\otimes_{\shd_Y^\hbar} \mathcal{K})\to \varprojlim_{n\geq 0} f_\ast(\shm_n\otimes_{\shd_Y}\shd_{Y\to X}),
\end{equation}
which defines a $\C^{\hbar}$-linear transformation of functors of stacks. When $f$ is a closed embedding, as proved in Corollary \ref{closedemb} below, it is an isomorphism of functors. Indeed, we don't know if it is an isomorphism in general, as explained in Remark \ref{R} below. However we have the following partial results:
\begin{lemma}\label{hcomplete}
If $\shm\in\text{Mod}_{coh}({\shd_Y^\hbar}^{op})$ is such that $\shm\otimes_{\shd_Y^\hbar} \mathcal{K}$ is $\hbar$-complete then  (\ref{eq:405}) is an isomorphism.
\end{lemma}

\begin{proof}
If $\shm\otimes_{\shd_Y^\hbar} \mathcal{K}$ is $\hbar$-complete then
$$\shm\otimes_{\shd_Y^\hbar} \mathcal{K}\simeq \varprojlim_{n\geq 0}(\shm\otimes_{\shd_Y^\hbar} \mathcal{K})_n=\varprojlim_{n\geq 0}(\C_{Y,n}^\hbar\otimes_{\C_Y^\hbar}(\shm\otimes_{\shd_Y^\hbar} \mathcal{K}))\simeq$$ $$\simeq \varprojlim_{n\geq 0}(\shm_n\otimes_{\shd_Y^\hbar} \mathcal{K})\simeq\varprojlim_{n\geq 0}(\shm_n\otimes_{\shd_Y}\shd_{Y\to X}),$$ and since $f_*$ commutes with projective limits the result follows.
\end{proof}

\begin{lemma}\label{chc2}
 For any $\shm\in\text{Mod}_{coh}({\shd_Y^\hbar}^{op})$, $\shm\otimes_{\shd_Y^\hbar} \mathcal{K}$ is  $c\hbar c$. \end{lemma}

\begin{proof}
By Proposition \ref{KK}, $\mathcal{K}$ is $c\hbar c$. On the other hand we can choose  a local presentation of $\shm$ by locally free $\shd_Y^\hbar$-modules to which we apply the right exact functor $\cdot\otimes_{\shd_Y^{\hbar}} \mathcal{K}$. Hence   $\shm\otimes_{\shd_Y^\hbar} \mathcal{K}$ is locally the cokernel of a $\C_Y^{\hbar}$-linear morphism of $chc$-modules and the result follows by Lemma \ref{chc}.
\end{proof}

\begin{corollary}
Let $\shm\in\text{Mod}_{coh}({\shd_Y^\hbar}^{op})$. Then  (\ref{eq:405}) is an isomorphism in each one of the following cases:

\begin{itemize}
\item[(i)] $\shm$ is an $\hbar$-torsion module;

\item[(ii)] $\shm\otimes_{\shd_Y^\hbar} \mathcal{K}$ is $\hbar$-torsion free.
\end{itemize}
\end{corollary}

\begin{proof}
By Lemma  \ref{hcomplete} it is enough to prove that in both cases  $\shm\otimes_{\shd_Y^\hbar} \mathcal{K}$ is $\hbar$-complete.

(i) If $\shm$ is an $\hbar$-torsion module and since the result is local, we may assume that there exists some $N>0$ such that $\hbar^N\shm=0$. This implies that $\hbar^N(\shm\otimes_{\shd_Y^\hbar} \mathcal{K})\simeq 0$ and, in particular, that $\shm\otimes_{\shd_Y^\hbar} \mathcal{K}$ is $\hbar$-complete.

(ii) If $\shm\otimes_{\shd_Y^\hbar} \mathcal{K}$ is an $\hbar$-torsion free module,  by Lemma \ref{chc2} together with \cite{KS2}, Lemma 1.5.4, it is $\hbar$-complete.
\end{proof}
The following result is possibly well known but we find it useful to prove  here:

\begin{lemma}\label{L2}
Let $F\in\text{Mod}_{\R-c}(\C_X^\hbar)$. Then $F$ is $\hbar$-complete.
\end{lemma}

\begin{proof}
We shall prove that the natural morphism $F\to\underset{n\geq 0}{\varprojlim} F_n$ is an isomorphism.

By the triangulation theorem (Proposition 8.2.5 of \cite{KS1}) we may assume that $F$ is a constructible sheaf on the realization of a finite simplicial complex $(S, \Delta)$ (we refer to \cite{KS1} for the notation) and, for each $n$,
 $F_n$ being the cokernel of the morphism $\hbar^{n+1}:F\to F$, it is also constructible on $(S, \Delta)$. It follows that there exists a locally finite open covering $\{U(\sigma)\}_{\sigma \in\Delta}$ of $S$  such that, for each $\sigma\in\Delta$ and $x\in |\sigma|$,   $\Gamma(U(\sigma); F)\simeq F_x$ and $\Gamma(U(\sigma); F_n)\simeq (F_n)_x$, for every $n\in \N$.

As   a finitely generated  $\C^\hbar$-module,  $F_x$ is $\hbar$-complete and hence $$\Gamma(U(\sigma); F)\simeq F_x\simeq \varprojlim_{n\geq 0}(F_x)_n\simeq\varprojlim_{n\geq 0}(F_n)_x\simeq \varprojlim_{n\geq 0}\Gamma(U(\sigma); F_n)\simeq \Gamma(U(\sigma); \varprojlim_{n\geq 0} F_n),$$ and the result follows.
\end{proof}

\begin{corollary}\label{hol}
If $\shm\in\text{Mod}_{coh}({\shd_Y^\hbar}^{op})$ is holonomic, then $\shm\otimes_{\shd_Y^\hbar} \sho_Y^\hbar$ is $\hbar$-complete. In particular, when $f:Y\to\{pt\}$,  (\ref{eq:405}) is an isomorphism for every holonomic $\shd_Y^\hbar$-module $\shm$.
\end{corollary}

\begin{proof}
Since in this case $\mathcal{K}\simeq \sho_Y^\hbar$, then $$\shm\otimes_{\shd_Y^\hbar} \mathcal{K}\simeq H^0( R\shh \text{om}_{\shd_Y^\hbar}(D^\prime_\hbar \shm, \sho_Y^\hbar))$$ is  $\R$-constructible by~\cite[Th.3.13]{DGS} and the result follows by Lemma~\ref{L2}.
\end{proof}

We then infer that (\ref{eq:405}) is an isomorphism if $Y$ is a complex line,  $\shm\in\text{Mod}_{coh}({\shd_Y^{\hbar}}^{op})$ has a discrete support and $f$ is the constant map $f:Y\to \{pt\}$. Indeed, as proved in \cite{DGS}, the support of $\shm$ coincides with the support
of $\shm_0$, so, if $\text{supp}\,(\shm)$ is discrete, $\shm$ is holonomic and the statement follows by Corollary \ref{hol}.

\begin{remark}\label{R}
As a matter of fact we didn't find a counter-example for the  conjecture that if $\shm$ is $\shd_Y^{\hbar}$-coherent, then $\shm\otimes_{\shd_Y^\hbar} \mathcal{K}$ is always $\hbar$-complete. Of course, such a counter-example, to exist,  should firstly occur in the smooth case. This difficulty prevented us from applying succefully our results to extend  the functor of    proper direct image except for a closed embedding as  above.
\end{remark}

\begin{corollary}\label{closedemb}
When $f$ is a local immersion, the morphism (\ref{eq:405}) is an isomorphism for every  $\shm\in\text{Mod}_{coh}({\shd_Y^\hbar}^{op})$.
\end{corollary}
\begin{proof}
By Theorem~\ref{T124}, $\mathcal{K}$ is flat over $\shd_Y^\hbar$. Moreover, as  can be checked by the reader, $(f^{-1}\shd_X)^{\hbar}$ is an algebra of formal deformation.

 Since $\mathcal{K}$ is coherent over  $(f^{-1}\shd_X)^{\hbar}$,   $\shm\otimes_{\shd_Y^{\hbar}}\mathcal{K}$ is coherent over $(f^{-1}\shd_X)^{\hbar}$ hence it is $\hbar$-complete.
Since $\hbar$-completeness is a local property the result follows.

\end{proof}

\subsubsection{An alternative extension}
The idea now is to use the transfer module $\mathcal{K}$ to mimic the $\shd$-module construction of direct images.
The $(\shd_Y^\hbar, f^{-1}(\shd_X^\hbar))$-bimodule structure on $\mathcal{K}$ allows us to define  functors $$R\underline{f}_{\ast}^\hbar,R\underline{f}_{!}^\hbar:D^b({\shd_Y^\hbar}^{op}) \to D^b({\shd_X^\hbar}^{op})$$ respectively of direct image and of proper direct image, by:
\begin{eqnarray*}
&&R\underline{f}_{\ast}^\hbar(\shm):= Rf_\ast (\shm\stackrel{L}{\otimes}_{\shd_Y^\hbar} \mathcal{K}),\\
&&R\underline{f}_{!}^\hbar(\shm):= Rf_! (\shm\stackrel{L}{\otimes}_{\shd_Y^\hbar} \mathcal{K}).
\end{eqnarray*}

We remark that Corollary \ref{closedemb} implies that these definitions coincide with $\underline{i^{\hbar}}_{\ast}$ for a closed embedding $i$.

\begin{lemma}\label{L:404}
Let $\shm\in D^b_{coh}({\shd_Y^\hbar}^{op})$. Then $R\underline{f}_{\ast}^\hbar(\shm)$ is $c\hbar c$.
\end{lemma}
\begin{proof}
Recall that $\mathcal{K}$ is $c\hbar c$ by Prop.~\ref{KK}. Consider the canonical  isomorphisms  in $D^b(\C^\hbar_X)$:
$$\shm\stackrel{L}{\otimes}_{\shd_Y^\hbar}\mathcal{K}\simeq R\shh \text{om}_{\shd_Y^\hbar}(D^\prime_\hbar \shm, \shd_Y^\hbar)\stackrel{L}{\otimes}_{\shd_Y^\hbar}\mathcal{K}\simeq R\shh \text{om}_{\shd_Y^\hbar}(D^\prime_\hbar \shm, \mathcal{K}).$$
Hence, $\shm\stackrel{L}{\otimes}_{\shd_Y^\hbar}\mathcal{K}$ is also $c\hbar c$.
Finally, we conclude that $R\underline{f}_{\ast}^\hbar(\shm)$ is $c\hbar c$ by Proposition \ref{P125}.
\end{proof}

We are now able to extend to $\shd^\hbar$-modules the classical coherence criterion of direct images of $\shd$-modules:

\begin{theorem}\label{T:405}
Suppose that $\shm\in D^b_{good}({\shd_Y^\hbar}^{op})$ (resp. $\shm\in D^b_{hol}({\shd_Y^\hbar}^{op})$)
and that $f$ is proper on $supp(\shm)$.
Then,  $$R\underline{f}_{\ast}^\hbar(\shm)\in D^b_{good}({\shd_X^\hbar}^{op})$$
(resp. $R\underline{f}_{\ast}^\hbar(\shm)\in D^b_{hol}({\shd_X^\hbar}^{op}).)$
\end{theorem}
\begin{proof}
Since by assumption  $gr_\hbar(R\underline{f}_{\ast}^\hbar(\shm))$ is an object of $D^b_{good}({\shd_X}^{op})$ (resp. $D^b_{hol}({\shd_X}^{op})$),
the conclusion follows by applying Theorem~\ref{T123} to the object $R\underline{f}_{\ast}^\hbar(\shm)\in D^b({\shd_X^\hbar}^{op})$.
\end{proof}

\subsection{Review on specialization, vanishing cycles and nearby-cycles.}
\subsubsection{Review on Sato's specialization, vanishing and nearby-cycles.}
We refer to Chapters IV and VIII of ~\cite{KS1} for a detailed study of the constructions  below.

Let $X$ be a complex analytic manifold 
and
$Y$ a submanifold of codimension $d$. Recall that $\mathbb{K}$ denotes a unital commutative Noetherian ring with finite global dimension.
Denote by $\widetilde{X}_Y$ the normal deformation of $Y$ in $X$.
This is a real analytic manifold endowed with two canonical maps $p:\widetilde{X}_Y\to X$
and $t:\widetilde{X}_Y\to\R$  such that $T_YX$ is identified to the real analytic  hypersurface of $\widetilde{X}_Y$ given by the equation $t=0$.

Denote by $s:T_YX\hookrightarrow \widetilde{X}_Y$ the canonical embedding.
Set $\Omega=t^{-1}(\R^+)\stackrel{j}{\hookrightarrow}\widetilde{X}_Y$ and
denote by $\widetilde{p}$ the restriction of $p$ to $\Omega$.


 Recall that the  Sato's specialization functor on $D^b(\mathbb{K}_X)$ is given by $F\mapsto \nu_Y^{\mathbb{K}}(F)\mathbin{:=} s^{-1} \R j_\ast \widetilde{p}^{-1} F$.
Recall that $\nu_Y^{\mathbb{K}}$ induces a functor: $D^b_{\C-c}(\mathbb{K}_X)\to D^b_{\C-c}(\mathbb{K}_{T_YX})$.\newline

Let us now assume that $Y$ is a complex closed smooth hypersurface of $X$ given as the zero locus of a holomorphic function $f:X\to \C$.

Let $\widetilde{\C}^*$ be the universal covering of $\C^*=\C\backslash\{0\}$ and let $p:\widetilde{\C}^*\to\C^*$ be the projection. Denote by $\widetilde{X}^*$ the fibered product $X\times_\C \widetilde{\C}^*$ and let $\widetilde{p}$ be the projection associated to $\text{id}\times_\C p$:
\begin{eqnarray*}
&&\xymatrix{
& \widetilde{X}^*\ar[r] \ar[d]^-{\widetilde{p}} & \widetilde{\C}^*\ar[d]^-p  \\
Y\ar[r]_-i & X\ar[r]_-f & \C .}
\end{eqnarray*}

Recall that the nearby-cycle functor $\psi_Y^{\mathbb{K}}: D^b(\mathbb{K}_X)\to D^b(\mathbb{K}_Y)$ is defined by: $$\psi_Y^{\mathbb{K}}(F):=i^{-1}R\widetilde{p}_*\widetilde{p}^{-1}(F),$$ and that the vanishing-cycle functor $\phi_Y^{\mathbb{K}}: D^b(\mathbb{K}_X)\to D^b(\mathbb{K}_Y)$  assigns to an object $F\in D^b(\mathbb{K}_X)$ the mapping cone of  $i^{-1}F\to \psi_Y^{\mathbb{K}}(F)$. The natural morphism $\psi_Y^{\mathbb{K}}(F)\to \phi_Y^{\mathbb{K}}(F)$ is called the canonical morphism and denoted by $\text{can}$.  Both $\psi_Y^{\mathbb{K}}(F)$ and $\phi_Y^{\mathbb{K}}(F)$ induce functors $D^b_{\C-c}(\mathbb{K}_X)\to D^b_{\C-c}(\mathbb{K}_Y)$ .


\begin{lemma}\label{L31}
\begin{enumerate}[{\rm (a)}]
\item If $g:X \to Z$ is a morphism of complex manifolds and $F\in D^b(\C^\hbar_X)$,
then in $D^b(\C_Z)$,
$\R g_\ast^{\C^\hbar} F\simeq\R g_\ast^\C F$ and $\R g_!^{\C^\hbar} F\simeq \R g_! ^\C F$.
\item For every $F\in D^b(\C^\hbar_X)$, we have $\nu_Y^{\C^\hbar}(F)\simeq\nu_Y^\C(F)$ and, if $Y$ is a smooth hypersurface of $X$, we also have $\psi^{\C^\hbar}_Y(F)\simeq \psi^{\C}_Y(F)$.

\end{enumerate}
\end{lemma}
\begin{proof}
Let $I^\bullet$ be a flabby resolution of $F$ in $D^b(\C^\hbar_X)$.
Then each $I^j$ is also flabby in $\text{Mod}(\C_X)$.
Hence, both $\R g_\ast^{\C^\hbar} F$ and $\R g_\ast^{\C} F$ are quasi-isomorphic to $g_\ast(I^\bullet)$.
Similarly, using a c-soft resolution of $F$ instead, we get $\R g_!^{\C^\hbar} F\simeq \R g_! ^\C F$,
which proves (a). (b) follows as a consequence of (a).
\end{proof}

 Henceforth we keep the  notations $\nu_Y, \psi_Y, \phi_Y$ for the specialization or the nearby-cycle/vanishing-cycle functors on sheaves of $\C^\hbar$-modules.

\begin{remark}
Given $F\in D^b(\C^\hbar_X)$, since $gr_{\hbar}$ commutes with inverse and direct (proper direct) image, we conclude that $gr_{\hbar}(\nu_Y(F))\simeq \nu_Y(gr_{\hbar}(F))$, $gr_{\hbar}(\phi_Y(F))\simeq \phi_Y(gr_{\hbar}(F))$ and $gr_{\hbar}(\psi_Y(F))\simeq \psi_Y(gr_{\hbar}(F))$.
 \end{remark}

\subsubsection{Review on specialization, vanishing and nearby-cycle functors for $\shd$-modules.}

We start by recalling the (exact) functor of specialization of $\shd_X$-modules (along a submanifold) as developed in the  work of
M.~Kashiwara (\cite{Ka1}). For the basic material besides~\cite{Ka1}, we refer to~\cite{LM},~\cite{MM} and~\cite{MF}.

Let $Y\subset X$ be a submanifold of $X$  and denote by $I$ the defining ideal of $Y$
and by $\pi:T_Y X\to Y$ the projection of the normal bundle to $Y$.
One denotes by ${V_Y}^\bullet(\shd_X)$ (or by $V^\bullet$ for short once $Y$ is fixed)
Kashiwara-Malgrange $V$-filtration of $\shd_X$ with respect to $Y$:
\begin{equation*}
V^k(\shd_X)=\left\{P\in\shd_X: P(I^j)\subset I^{j+k}, \forall j, k\in\mathbb{Z} \ j,j+k\geq 0\right\}
\end{equation*}
 The graduate ring $gr_V(\shd_X)$
is isomorphic to $\pi_\ast\shd_{\left[T_YX\right]}$,
where $\shd_{\left[T_YX\right]}$ denotes the sheaf of homogeneous differential operators over $T_Y X$.

A coherent $\shd_X$-module $\shm$ always admits locally a good $V$-filtration.

Denote by $\theta$ the Euler field on $T_YX$.

Recall that a coherent $\shd_X$-module $\shm$ is specializable along
$Y$ if for every local good $V$-filtration $V^{\bullet}(\shm)$ on $\shm$
there is locally a non zero polynomial $b\in\C[s]$ such that
$$
b(\theta-k) V^k(\shm)\subset V^{k+1}(\shm),\quad \forall k\in\Z;
$$
 $b$ is called a Bernstein-Sato polynomial or a $b$-function
associated to the filtration $V^{\bullet}$.

In the sequel, when there is no risk of confusion,
we often write {\em specializable} instead of {\em specializable along $Y$,}
once the submanifold $Y$ is fixed.

Denote by $G$ a section of the canonical morphism $\C\to\C/\Z$ and fix on $\C$ the lexicographical order. Let $\shm$ be a specializable $\shd_X$-module and denote by $V_G(\shm)$ a good $V$-filtration of $\shm$ admitting locally $b$-function whose zeros are contained in $G$. Such condition defines a global filtration (Kashiwara's canonical V-filtration) on $\shm$ which  is uniquely defined.

The specialized of $\shm$ along $Y$ is the coherent $\shd_{T_Y X}$-module:
$$
\nu_Y(\shm)= \shd_{T_Y X}\otimes_{\shd_{\left[T_Y X\right]}} \pi^{-1} gr_{V_G}(\shm),
$$
and this definition doesn't depend on the choice of $G$.

\begin{remark}\label{R2}
Let us fix $G$ as above. Given an exact sequence of specializable $\shd_X$-modules: $$0\to\shm_1\to\shm\to\shm_2\to 0,$$ if $b_i(s)$ is a local Bernstein-Sato polynomial for the canonical $V$-filtration on $\shm_1$, $i=1,2$ then $b_1(s)\cdot b_2(s)$ is a  Bernstein-Sato polynomial for the canonical $V$-filtration on $\shm$ (see for example~\cite[Prop.4.2]{MM}).\end{remark}


Denote by $\text{Mod}_{sp}(\shd_X)$  the full Serre subcategory of $\text{Mod}_{coh}(\shd_X)$
of specializable $\shd_X$-modules along $Y$. The assignment $U\mapsto \text{Mod}_{sp}(\shd_X|_U)$ defines a full Serre substack of $\mathscr{M}od_{coh}(\shd_X)$.

The correspondence $\shm\mapsto\nu_Y(\shm)$ determines an exact functor from $\text{Mod}_{sp}(\shd_X)$
to $\text{Mod}_{coh}(\shd_{T_YX})$.

Let us suppose now that $Y$ is a complex closed smooth hypersurface of $X$ given by the zero locus of a holomorphic function $f:X\to \C$.
Recall that in this case, we can also associate to a specializable $\shd_X$-module $\shm$  the nearby-cycle module $$\psi_Y(\shm)\simeq gr_{V_G}^0(\shm)=\frac{V_G^0(\shm)}{V_G^{1}(\shm)},$$ and the vanishing-cycle module $$\phi_Y(\shm)\simeq gr_{V_G}^{-1}(\shm)=\frac{V_G^{-1}(\shm)}{V_G^{0}(\shm)}.$$ Thus $\psi_Y, \phi_Y: \text{Mod}_{sp}(\shd_X)\to\text{Mod}_{coh}(\shd_Y)$ are exact functors.

\subsection{Specialization, vanishing cycles and nearby-cycles for $\shd_X^\hbar$-modules.}

Let $Y$ be a submanifold of a complex manifold $X$.
 According to the preceding subsection, we fix  a section $G$ of the canonical morphism $\C\to\C/\Z$  to which  all canonical $V$-filtrations mentioned below will refer.

Given $\shm\in\text{Mod}_{coh}(\shd_{X,n}^{\hbar})$ we  say that $\shm$ is specializable along $Y$  and denote it by $\shm\in \text{Mod}_{sp}(\shd_{X,n}^{\hbar})$ if it is so  when endowed with the structure of $\shd_X$-module  explained in  Remark \ref{R:St}. We obtain a full Serre substack $\mathcal{S}$ of $\mathscr{M}od_{coh}(\shd_X^{\hbar})$
by assigning to each open subset $U\subset X$ the full Serre subcategory  $\mathcal{S}(U)=\cup_{n\geq 0} \text{Mod}_{sp}(\shd_{X,n}^{\hbar}|_U)$.

\begin{definition}\label{D5}
We say that a coherent $\shd_X^\hbar$-module $\shm$ is specializable along  $Y$ if $\shm\in \text{Mod}_{\mathcal{S}}(\shd_X^{\hbar})$.
\end{definition}

Equivalently, $gr_\hbar(\shm)$ is specializable in the $\shd_X$-modules sense, that is, both $_0\shm$ and $\shm_0$ are
specializable $\shd_X$-modules along  $Y$.

\begin{example}\label{E6}

Every coherent $\shd_X^\hbar$-module $\shm$ such that $supp(\shm)\subset Y$ is specializable along $Y$. Indeed we have
\begin{equation*}
supp(gr_\hbar(\shm))=supp(_0\shm)\cup supp(\shm_0)\subset supp(\shm)\subset Y.
\end{equation*}
Hence $_0\shm$ and $\shm_0$ are specializable along $Y$.
\end{example}

In the sequel, for short, we shall often say that $\shm$ is specializable omitting the
reference to the submanifold $Y$.

We denote by $\text{Mod}_{sp}(\shd_X^\hbar)$ the category $\text{Mod}_{\mathcal{S}}(\shd_X^\hbar).$

As a functor $\Phi:Op(X)\to Op(T_YX)$ satisfying assumption \ref{E:fi} we consider the data $U\mapsto \Phi(U)=\pi^{-1}(U\cap Y)$ where $\pi:T_YX\to Y$ denotes the projection.
According to  Theorem \ref{T:1} we are  in the conditions  to extend (uniquely up to an isomorphism) the exact functor
 $$\nu_Y: \text{Mod}_{sp}(\shd_X)\to \text{Mod}_{coh}(\shd_{T_YX})$$ as an exact functor $$\begin{matrix}\nu_Y^\hbar:&\text{Mod}_{sp}(\shd_X^\hbar)&\to&\text{Mod}_{\text{coh}}(\shd_{T_Y X}^\hbar)\\ & \shm&\mapsto &\nu_Y^\hbar(\shm):=\underset{{n\geq 0}}\varprojlim \nu_Y(\shm_n).\end{matrix}$$

\begin{definition}
Given $\shm\in\text{Mod}_{sp}(\shd_X^\hbar)$,   we shall say that $\nu_Y^\hbar(\shm)$ is the specialized of $\shm$ (along $Y$).
\end{definition}

Propositions \ref{P:8} and~\ref{G20} entail the following result:

\begin{corollary}\label{P8}
Let $\shm$ be an $\hbar$-torsion $\shd_X^\hbar$-module.
Then $\shm$ is specializable as a $\shd_X^\hbar$-module
if and only if $\shm$ is specializable in the $\shd_X$-modules sense.
Moreover, if $\shm$ is specializable then $\nu_Y^\hbar(\shm)\simeq \nu_Y(\shm)$ in $\text{Mod}_{coh}(\shd_{T_YX})$.
\end{corollary}

By Proposition~\ref{G9} we have the following characterization:

\begin{corollary}\label{P9}
Let $\shm$ be a coherent $\shd_X^\hbar$-module.
Then the following properties are equivalent:
\begin{enumerate}
\item $\shm$ is a specializable $\shd_X^\hbar$-module;
\item $\shm_0$ is a specializable $\shd_X$-module;
\item $\shm_n$ is specializable as a $\shd_X$-module, for each $n\geq 0$.
\end{enumerate}
\end{corollary}

\begin{remark}\label{R101}
Let $\shm$ be  a specializable $\shd_X^{\hbar}$-module.  Regarding $gr_\hbar(\shm)$ as an object of $D^b(\shd_X)$, we have a specializable complex in the sense of \cite{LM}. Since  $$gr_{\hbar}\nu_Y^{\hbar}(\shm)\simeq \underset{n\geq 0}{\varprojlim}gr_{\hbar}\nu_Y(\shm_n)$$
and, for each $n$, by construction, $gr_{\hbar}\nu_Y(\shm_n)$ is isomorphic to $\nu_Ygr_{\hbar}(\shm_n)$, we get a morphism  $$gr_{\hbar}\nu_Y^{\hbar}(\shm)\to \nu_Ygr_{\hbar}(\shm).$$ Theorem \ref{T:1} asserts that this morphism is an isomorphism in $D^b(\shd_{T_YX})$.
\end{remark}

\begin{remark}\label{R11}
Since the ring $\shd_X^\hbar$ is not filtered neither by the order nor by $V$-filtrations,   the notion of Bernstein polynomial for a specializable $\shd_X^{\hbar}$-module does not make sense in general.  However, we have the following result:

\end{remark}

\begin{proposition}\label{L12}
Let $\shm$ be a specializable $\hbar$-torsion free$\shd_X^\hbar$-module.
Assume that $b(s)$ is a Bernstein polynomial for the canonical $V$-filtration  on $\shm_0$
as a specializable $\shd_X$-module.
Then, $b_n(s):=(b(s))^{n+1}$ is a Bernstein polynomial of $\shm_n$ for the canonical  $V$-filtration.
\end{proposition}
\begin{proof}
The sequence $$0\to\shm_{0}\xrightarrow{\overline{\hbar}}\shm_1\xrightarrow{\rho_{0,1}} \shm_{0}\to 0,$$ together with Remark~\ref{R2} entails that, if $b(s)$ is a Bernstein polynomial for the canonical $V$-filtration on $\shm_0$, then $(b(s))^2$ is a Bernstein polynomial for the canonical $V$-filtration on $\shm_1$, and we proceed by induction  applying the same argument to the sequence $$0\to\shm_{n-1}\xrightarrow{\overline{\hbar}}\shm_n\xrightarrow{\rho_{0,n}} \shm_{0}\to 0.$$
\end{proof}

In the examples below we assume  $X=\mathbb{C}^m$, for some $m\in\mathbb{N}$, with coordinates $(t, x_1,..., x_{m-1})$, and  $Y=\{(t, x_1,..., x_{m-1})\in\mathbb{C}^m: t=0\}$.
\begin{example}
Let $\shm$ be a $\shd_X^{\hbar}$-module with one generator, let us say $\shm\simeq \shd_X^{\hbar}/\mathcal{J}$, for a coherent ideal $\mathcal{J}$. Then  we have a chain of isomorphisms of $\shd_X$-modules, $$\shm_n\simeq \frac{\shd_X^{\hbar}}{\hbar^{n+1}\shd_X^{\hbar}+\mathcal{J}}\simeq\frac{\oplus_{i=0,..., n}\shd_X \hbar^i}{\tilde{\mathcal{J}}_n},$$
where $\tilde{\mathcal{J}}_n$ is the submodule of $\oplus_{i=0,...,n}\shd_X \hbar^i$ given by
$$\tilde{\mathcal{J}}_n=\frac{\mathcal{J}}{\hbar^{n+1}\shd_X^{\hbar}\cap\mathcal{J}}.$$
Suppose that $\shm=\shd_X^\hbar/\shd_X^{\hbar}b(t\partial_t)$, where $b(s)$ is a polynomial in $\C^\hbar[s]$, $b(s)=\sum_{i=0}^ma_i(\hbar)s^i$,  for some $m\in\mathbb{N}$ and, for $i\geq 0$, $a_i(\hbar):=\sum_{j\geq 0} a_{ij}\hbar^j\in \C^\hbar$. Set $b_0(s)=\sum_{i=0}^m a_{i0}s^i$.

Since  $\shm_0\simeq \shd_X/\shd_X b_0(t\partial_t)$, $\shm$   is  specializable if and only if $b_0(s)$ is a non zero polynomial in $\C[s]$. We shall calculate  particular cases in the following examples:
\end{example}
\begin{example}\label{E13}
Let
$\shm=\shd_X^\hbar/\shd_X^{\hbar}(\hbar t\partial_t+1)$. Clearly $\shm_0=0$ hence $\shm_n=0$ for every $n$, which entails $\nu_Y^{\hbar}(\shm)=0$.
\end{example}
\begin{example}\label{E12}
Assume  that $\mathcal{J}=\shd_X^{\hbar}(t\partial_t-\hbar)$. Then $$\tilde{\mathcal{J}}_n\simeq \{P_0 t\partial_t+\sum_{i=1}^{n}(P_it\partial_t-P_{i-1}) \hbar^i: P_i\in\shd_X\}.$$

Therefore $\shm_n$ can be identified with the cokernel of the $\shd_X$-linear morphism from $\shd_X^{n+1}$ to itself given by the right multiplication by the matrix
$$A_n=\left[\begin{array}{cccccc} t\partial_t & -1  &0& ... &0 &0 \\ 0&t\partial_t&-1&...& 0&0\\ ...&...&...&...&...&...\\ 0& 0& 0&...  &  t\partial_t& -1 \\  0& 0 & 0&... &0 &t\partial_t \end{array}\right].$$   Denoting by $u_{1,n},...,u_{n+1,n}$, respectively, the classes of the elements of the canonical basis of $\shd_X^{n+1}$ in $\shm_n$,  we obtain a system of generators for $\shm_n$ satisfying $$(t\partial_t)u_{1,n}=0,\, (t\partial_t)u_{k,n}=u_{k-1,n},$$  for $k=2,...,n+1$. Classically one derives an isomorphism $$\shd_X/\shd_X(t\partial _t)^{n+1}\to \shm_n$$ defined by $$1 \,\text{mod} \,\shd_X(t\partial _t)^{n+1} \mapsto u_{n+1,n}.$$
Therefore, denoting by $(x,\tau)$  the associated coordinates in $T_YX$, we obtain $\nu_Y(\shm_n)\simeq
\shd_{T_YX}/\shd_{T_YX}(\tau\partial _{\tau})^{n+1}$.

Since $t\partial_t$ acts by multiplication by $\hbar$ in $\shm_n$,  the action of $\hbar$ in $\nu_Y(\shm_n)$ coincides with the multiplication by $\tau\partial_{\tau}$ hence, as a $\shd^{\hbar}_{T_YX}$-module,  $$\nu_Y(\shm_n)\simeq \frac{\shd^{\hbar}_{T_YX}}{\shd^{\hbar}_{T_YX}(\tau\partial_{\tau}-\hbar)+\hbar^{n+1}\shd^{\hbar}_{T_YX}}$$ and it follows that $$\nu_Y^\hbar(\shm)=\varprojlim_{n\geq 0} {\nu_Y}(\shm_n)\simeq \frac{\shd_{T_YX}^{\hbar}}{\shd_{T_YX}^{\hbar}(\tau\partial_{\tau} -\hbar)}.$$
\end{example}

Assume now that $Y$ is a complex closed smooth hypersurface of $X$ given by the zero locus of a holomorphic function $f:X\to \C$.
We can  extend the exact functors
$\psi_Y, \phi_Y:\text{Mod}_{\text{sp}}(\shd_X)\to\text{Mod}_{\text{coh}}(\shd_Y)$, respectively, as functors $$\begin{matrix}\psi_Y^\hbar:&\text{Mod}_{\text{sp}}(\shd_X^\hbar)&\to&\text{Mod}_{\text{coh}}(\shd_{Y}^\hbar)\\ & \shm&\mapsto &\psi_Y^\hbar(\shm):=\underset{{n\geq 0}}\varprojlim \psi_Y(\shm_n),\end{matrix}$$ $$\begin{matrix}\phi_Y^\hbar:&\text{Mod}_{\text{sp}}(\shd_X^\hbar)&\to&\text{Mod}_{\text{coh}}(\shd_{Y}^\hbar)\\ & \shm&\mapsto &\phi_Y^\hbar(\shm):=\underset{{n\geq 0}}\varprojlim \phi_Y(\shm_n).\end{matrix}$$

One can rewrite Propositions~\ref{G20} and ~\ref{G18} and Corollaries~\ref{C171} and~\ref{C173}  replacing the functor $F$ respectively by $\psi_Y$ and $\phi_Y$.

\begin{example}
Keeping the notations of examples above, we infer from the results of~\cite{MF} that, for each $n\geq 0$, $\psi_Y(\shm_n)$ is quasi-isomorphic to the complex $\nu_Y(\shm_n)\xrightarrow{\tau-1}\nu_Y(\shm_n),$ that is, $$\psi_Y(\shm_n)\simeq \frac{\shd_{T_YX}}{\shd_{T_YX}(\tau-1) +\shd_{T_YX}(\tau\partial_\tau)^{n+1}}\simeq \shd_Y^{n+1}.$$ Thus $$\psi_Y^\hbar(\shm)\simeq  \frac{\shd_{T_YX}^\hbar}{\shd_{T_YX}^\hbar(\tau-1) +\shd_{T_YX}^\hbar(\tau\partial_\tau-\hbar)},$$
in other words $\psi_Y^\hbar(\shm)$ is quasi-isomorphic to the complex $\nu_Y^\hbar(\shm)\xrightarrow{\tau-1}\nu_Y^\hbar(\shm)$.
\end{example}

\subsubsection{The (regular) holonomic case.}

Consider the Serre subcategory $\mathcal{S}$  of holonomic (respectively regular holonomic) $\shd_X$-modules.

Similarly to  \cite{DGS} for the case $n=0$, we see by  Proposition~\ref{G9}
that if $\shm$ is a holonomic (respectively, regular holonomic)  $\shd_X^\hbar$-module, then each $\shm_n$ is a holonomic (regular holonomic)  for the $\shd_X$-module structure of $\shd_{X,n}^{\hbar}$ given in Remark  \ref{R:St}.

Recall that every holonomic $\shd_X$-module is specializable along any submanifold $Y$, and the specialized module is also a holonomic module.
Similarly, we have:

\begin{corollary}\label{P23}
Any holonomic $\shd_X^\hbar$-module $\shm$ is specializable along any submanifold $Y$.
Moreover $\nu_Y^\hbar(\shm)$ is a holonomic $\shd_{T_Y X}^\hbar$-module. If $\shm$ is regular holonomic, so is $\nu_Y^\hbar(\shm)$.

When $Y$ is a smooth hypersurface, if $\shm$ is  holonomic (resp. regular holonomic),  $\psi_Y^\hbar(\shm)$ and $\phi_Y^{\hbar}(\shm)$ are  holonomic (resp. regular holonomic) as $\shd^{\hbar}_Y$-modules.
\end{corollary}

\subsubsection{Comparison Theorems.}

Let us recall that Kashiwara constructed in~\cite[Th.1]{Ka1}
for a regular holonomic $\shd_X$-module $\shm$ (or, more generally, for an object of $D^b_{rh}(\shd_X)$), canonical isomorphisms in $D^b(\C_{T_Y X})$
$$\begin{cases}Sol_{\shd_{T_Y X}}(\nu_Y(\shm))\isoto\nu_Y(Sol_{\shd_X}(\shm)), \\
DR_{\shd_{T_Y X}}(\nu_Y(\shm))\isofrom \nu_Y(DR_{\shd_X}(\shm)),
\end{cases}$$
and, when $Y$ is a smooth hypersurface of $X$, canonical isomorphisms in $D^b(\C_Y)$
$$\begin{cases}
Sol_{\shd_{Y}}(\psi_Y(\shm))\isoto\psi_Y(Sol_{\shd_X}(\shm)), \\
Sol_{\shd_{Y}}(\phi_Y(\shm))\isoto\phi_Y(Sol_{\shd_X}(\shm))
\end{cases}$$
and
$$\begin{cases}
DR_{\shd_{Y}}(\psi_Y(\shm))\isofrom \psi_Y(DR_{\shd_X}(\shm)), \\
DR_{\shd_{Y}}(\phi_Y(\shm))\isofrom \phi_Y(DR_{\shd_X}(\shm)).
\end{cases}$$

More precisely, setting:
\begin{eqnarray*}
&&\varphi_1=\nu_Y\circ DR:\text{Mod}_{rh}(\shd_X)\to D^b(\C_{T_Y X}),\\
&&\varphi_2=DR\circ\nu_Y:\text{Mod}_{rh}(\shd_X)\to D^b(\C_{T_Y X})
\end{eqnarray*}

Kashiwara's construction gives a natural transformation $\varphi_1 \stackrel{\Psi_K}{\to} \varphi_2$.
In particular, if the modules  are provided with a $\C^{\hbar}$ action, for any  $f\in\C^\hbar$
we get a commutative diagram in $D^b(\C_{T_YX})$:
\begin{eqnarray*}
&&\xymatrix{
\nu_Y(DR(\shm))\ar[r]^-{f}\ar[d]^-{\Psi_K(\shm)}&\nu_Y(DR(\shm))\ar[d]^-{\Psi_K(\shm)}\\
DR(\nu_Y(\shm))\ar[r]^-{f}&DR(\nu_Y(\shm))}
\end{eqnarray*}

As a consequence, $\Psi_K$ is $\C^\hbar$-linear.

We shall now generalize these isomorphisms to the $\hbar$-setting.

Denote by $\Omega_X^j$ the sheaf of holomorphic forms of degree $j$ on $X$.
Consider the De Rham complex of $X$:
\begin{eqnarray*}
0\to\Omega_X^0\stackrel{d}{\to}\Omega_X^1\to \cdots \to \Omega_X^{n-1}\stackrel{d}{\to}\Omega_X^n\to 0.
\end{eqnarray*}
Here $d$ denotes the usual exterior derivatives.

For a $\shd_X^{\hbar}$-module $\shm$ we have
$DR(\shm)\simeq \Omega_X^\bullet\otimes_{\sho_X}\shm^{\hbar}$
 in $D^b(\C_X^\hbar)$.
Indeed, each  $\Omega_X^j\otimes_{\sho_X}\shm$ has a natural structure of $\C_X^\hbar$-module
and the derivatives turn out to be $\C_X^\hbar$-linear.

By definition $\underset{{k\geq 0}}\varprojlim (\Omega_X^\bullet\otimes_{\sho_X}\shm_k)$ is  given by the complex:
\begin{eqnarray}\label{comparation1}
0\to \underset{{k\geq 0}}\varprojlim (\Omega_X^0\otimes_{\sho_X}\shm_k)\to
	\cdots \to \underset{{k\geq 0}}\varprojlim(\Omega_X^n\otimes_{\sho_X}\shm_k)\to 0.
\end{eqnarray}

\begin{lemma}\label{L30}
Let $\shm$ be a coherent $\shd_X^\hbar$-module.
Then  $\varprojlim_{k\geq 0}(\Omega_X^\bullet\otimes_{\sho_X}\shm_k)$
is isomorphic to $\Omega_X^\bullet\otimes_{\sho_X}\shm$ in $C^b(\C^\hbar_X)$.
\end{lemma}
\begin{proof}
Recall that $\shm$ is $\hbar$-complete.
For each $j$, the natural morphism
$$\Omega_X^j\otimes_{\sho_X}\shm\to\underset{{k\geq 0}}\varprojlim (\Omega_X^j\otimes_{\sho_X}\shm_k)$$
is an isomorphism because $\Omega_X^j$ is locally finitely free over $\sho_X$ and the projective limit is additive.
Clearly these isomorphisms are compatible with the derivatives hence \eqref{comparation1} is
isomorphic to  $\Omega_X^\bullet\otimes_{\sho_X}\shm$.
\end{proof}

\begin{theorem}\label{T:22}
For $\shm$ a regular holonomic $\shd_X^\hbar$-module,
there are canonical isomorphisms in $D^b_{\C-c}(\C_{T_Y X}^\hbar)$:
\begin{enumerate}[{\rm (i)}]
\item $DR_\hbar(\nu_Y^\hbar(\shm))\isofrom \nu_Y(DR_\hbar(\shm))$;
\item $Sol_\hbar(\nu_Y^\hbar(\shm))\isoto \nu_Y(Sol_\hbar(\shm))$.
\end{enumerate}
\end{theorem}
\begin{proof}

For each $k\geq 0$, we have a natural morphism in $D^b(\C^\hbar_{T_YX})$:
\begin{eqnarray*}
\nu_Y(DR_\hbar(\shm))\to\nu_Y(DR_\hbar(\shm_k)).
\end{eqnarray*}
Since $\shm_k$ is an $\hbar$-torsion regular holonomic $\shd_X^\hbar$-module,
we have $\nu_Y(DR_\hbar(\shm_k))\simeq\nu_Y(DR(\shm_k)) \simeq DR(\nu_Y(\shm_k))\simeq DR_{\hbar}(\nu^{\hbar}_Y(\shm_k))$
in $D^b(\C^\hbar_{T_YX})$.
In this way we get a canonical morphism in $D^b(\C^\hbar_{T_YX})$:
\begin{eqnarray}\label{morphism3}
\nu_Y(DR_\hbar(\shm))\to DR_\hbar(\nu_Y^\hbar(\shm_k))
\end{eqnarray}
which entail morphisms:
\begin{eqnarray*}
\nu_Y(DR_\hbar(\shm))\to\Omega_{T_YX}^\bullet\otimes_{\sho_{T_YX}} \nu_Y^\hbar(\shm_k).
\end{eqnarray*}
So we obtain a morphism in $C^b(\C^\hbar_{T_YX})$:
\begin{eqnarray*}
\nu_Y(DR_\hbar(\shm))\to\underset{{k\geq 0}}\varprojlim (\Omega_{T_Y X}^\bullet\otimes_{\sho_{T_YX}} \nu_Y^\hbar(\shm_k)).
\end{eqnarray*}

Finally
(i) follows from Lemma \ref{L30} as the composition of the sequence of morphisms below:
$$\nu_Y(DR_\hbar(\shm))\to\Omega_{T_YX}^\bullet\otimes_{\sho_{T_YX}}\nu_Y^{\hbar}(\shm)\xrightarrow[qis]{}DR_\hbar(\nu_Y^\hbar(\shm)).$$

Let us now prove that (i) is an isomorphism.
Note that
$\nu_Y(DR_\hbar(\shm))$ and $DR_\hbar(\nu_Y^\hbar(\shm))$ are both objects of $D^b_{\C-c}(\C^\hbar_{T_YX})$ hence they are $c\hbar c$.
Therefore, it is enough to prove that we obtain an isomorphism when we apply $gr_\hbar$ to (i).
We have on one hand:
$gr_\hbar(\nu_Y(DR_\hbar(\shm)))\simeq\nu_Y(DR(gr_\hbar(\shm))).$
Since $gr_\hbar(\shm)\in D^b_{rh}(\shd_X)$ we have $\nu_Y (DR(gr_{\hbar}\shm))\simeq DR\nu_Y(gr_{\hbar}\shm).$

On the other hand: $gr_\hbar DR_\hbar(\nu_Y^\hbar(\shm))\simeq DR gr_\hbar(\nu_Y^\hbar(\shm))\simeq DR\nu_Y(gr_{\hbar}\shm),$
 by Remark \ref{R101}.

To end the proof we remark that  $(ii)$ follows by the following chain of isomorphisms:
\begin{eqnarray*}
Sol_\hbar(\nu_Y^\hbar(\shm))&\simeq&D^\prime_\hbar(DR_\hbar(\nu_Y^\hbar(\shm)))\\
&\isoto& D^\prime_\hbar\nu_Y(DR_\hbar(\shm))\\
&\simeq& \nu_Y (D^\prime_\hbar(DR_\hbar(\shm)))\\
&\simeq&\nu_Y(Sol_h(\shm)),
\end{eqnarray*}
 where the first and fourth isomorphisms  follow from (\ref{eq11}),
the second follows by applying the contravariant functor $D^\prime$ to $(i)$ and the
 third follows by Proposition~8.4.13 of~\cite{KS1}.
\end{proof}

Similarly one proves:
\begin{corollary}\label{comparison cycles}
Let $Y$ be a smooth hypersurface of $X$ and $\shm$ a regular holonomic $\shd_X$-module $\shm$.
There are canonical isomorphisms in $D^b_{\C-c}(\C_{Y}^\hbar)$:
\begin{enumerate}[{\rm (i)}]
\item $\begin{cases}Sol_{\hbar_{Y}}(\psi_Y^\hbar(\shm))\isoto\psi_Y(Sol_{\hbar_X}(\shm)), \\ Sol_{\shd_\hbar}(\phi_Y^\hbar(\shm))\isoto\phi_Y(Sol_\hbar(\shm));\end{cases}$

\item $\begin{cases}DR_{\shd_\hbar}(\psi_Y^\hbar(\shm))\isofrom \psi_Y(DR_\hbar(\shm)), \\ DR_\hbar(\phi_Y^\hbar(\shm))\isofrom \phi_Y(DR_\hbar(\shm)).\end{cases}$
\end{enumerate}
\end{corollary}

\subsection{Review on Fourier transforms and microlocalization.}

\subsubsection{Review of Fourier-Sato transform and microlocalization of sheaves.}

Denote by $\R^+$ the multiplicative group of positive numbers, and suppose  a real or complex manifold $E$
endowed with an action of $\R^+$ is given. One denotes by $\text{Mod}_{\R^+}(\mathbb{K}_E)$ the full subcategory of $\text{Mod}(\mathbb{K}_E)$ consisting
of sheaves $F$ such that for any orbit $b$ of $\R^+$ in $E$, $F|_b$ is a locally constant sheaf.
One also denotes  by $D^+_{\R^+}(\mathbb{K}_E)$ the full subcategory of $D^+(\mathbb{K}_E)$ consisting of objects $F$ such that for all $j\in\Z$, $H^j(F)\in\text{Mod}_{\R^+}(\mathbb{K}_E)$.
An object of $D^+_{\R^+}(\mathbb{K}_E)$ is called a conic object.

Let now $E\xrightarrow{\pi} Y$ denote a holomorphic vector bundle on a complex analytic manifold $Y$
and let $E^\prime\xrightarrow{\tilde{\pi}}Y$ denote its dual bundle.
We will be particularly concerned with the cases where $Y$ is a submanifold of a manifold $X$, $E=T_Y X$ is the tangent bundle to $Y$ on $X$ and
$E^\prime=T_Y^\ast X$ is the cotangent bundle to $Y$ on $X$.

Let $F$ be a sheaf of $\C$-vector spaces over $E$.
One says that $F$ is monodromic if it is locally constant along the orbits $\C^\ast\eta$ for each $\eta\in E\backslash Y$.
The category of monodromic sheaves is a full abelian subcategory of $\text{Mod}(\C_E)$.
An object $F\in D^b(\C_E)$ is  monodromic if the sheaves $H^i(F)$ are monodromic for every $i\in\Z$.
We denote by $D^b_{\text{mon}}(\C_E)$ the full subcategory of $D^b(\C_E)$ formed by monodromic objects.

Denote by $p_1$ and $p_2$ the canonical projections from $E\times_Y E^\prime$ to $E$ and $E^\prime$, respectively,
and set $P=\left\{(x,y)\in E\times_Y E^\prime: \left\langle x,y\right\rangle \geq 0\right\}$.
The Fourier-Sato transform is the functor $\shf^{\mathbb{K}}:D^+_{\R^+}(\mathbb{K}_E)\to D^+_{\R^+}(\mathbb{K}_{E^\prime})$ defined by
$\shf^{\mathbb{K}}(F)\mathbin{:=} \R {p_2\ast} \circ \R \Gamma_{ P} \circ p_1^{-1} (F).$

\begin{lemma}\label{L50}
Let $F\in D^+_{\R^+}(\C^\hbar_E)$.
Then, $\shf^{\C^\hbar}(F)\simeq \shf^\C(F)$ in $D^+_{\R^+}(\C_{E^\prime})$.
\end{lemma}
\begin{proof}
Let $i:P\to E\times_Y E^\prime$ be the embedding of $P$.
Then $\R \Gamma_{ P}\simeq \R i_! i^!$.
Since $i^!$ and $ p_1^{-1}$ are exact functors, the result follows by Lemma~\ref{L31}.
\end{proof}

In particular, if $F\in D^b_{\text{mon}}(\C_E)$ then $\shf^\C(F)\in D^b_{\text{mon}}(\C_{E^\prime})$.

In the case $E=T_Y X$, $E^\prime=T_Y^\ast X$ and
$\shf^{\mathbb{K}}:D^+_{\R^+}(\mathbb{K}_{T_YX})\to D^+_{\R^+}(\mathbb{K}_{T_Y^\ast X})$,
the composition $\mu_Y^{\mathbb{K}}\mathbin{:=}\shf^{\mathbb{K}}\circ\nu_Y^{\mathbb{K}}: D^b(\mathbb{K}_X)\to D^b_{\R+}(\mathbb{K}_{T_Y^\ast X})$
is called the geometrical microlocalization (see~\cite{KS1}).

By Lemmas~\ref{L31} and~\ref{L50}, one has:

\begin{lemma}
For $F\in D^b(\C^\hbar_X)$ the objects $\mu_Y^{\C^\hbar}(F)$ and $\mu_Y^\C(F)$ are isomorphic in $D^b(\C_{T_Y^\ast X})$.
\end{lemma}

\subsubsection{Review on Fourier transform and microlocalization for $\shd$-modules.}

Denote by $\shd_{\left[E\right]}\subset \pi_\ast \shd_E$ the sheaf of differential operators polynomial in the fibers.
Let $\theta$ denote the Euler field on $E$.
A $\pi_\ast(\shd_E)$ or a $\shd_{\left[E\right]}$-left coherent module $\shn$ is
monodromic if $\shn$ is generated by local sections satisfying $b(\theta)u=0$
for some non-vanishing $b(\theta)\in\C[\theta]$.
We denote this category by $\text{Mod}_\text{mon}(\shd_{\left[E\right]})$,
a Serre subcategory of $\text{Mod}_{\text{coh}}(\shd_{\left[E\right]})$.

Consider the sheaf $\Omega_{E/Y}$ of relative differential forms to $\pi:E\to Y$.

One defines an exact functor $\shf$ (Fourier Transform, cf.~\cite{BMV}) from $\text{Mod}_\text{mon}(\shd_{\left[E\right]})$ to $\text{Mod}_\text{mon}(\shd_{\left[E^\prime\right]})$ setting for $\shn\in\text{Mod}_\text{mon}(\shd_{\left[E\right]})$:
\begin{eqnarray*}
\shf(\shn)\mathbin{:=}\Omega_{E/Y}\otimes_{\pi^{-1}\sho_Y}\shn.
\end{eqnarray*}

Recall  that for each $\shn\in\text{Mod}_\text{mon}(\shd_{[E]})$
one constructs canonical isomorphisms in $D^b(\C_{E'})$:
\begin{eqnarray}
&&\shf^\C(Sol(\shn))\simeq Sol(\shf(\shn))[\text{-codim} Y]; \label{Fourier1}\\
&&\shf^\C(DR(\shn))\simeq DR(\shf(\shn))[\text{-codim} Y]. \label{Fourier2}
\end{eqnarray}

Consider $E=T_Y X$, $E^\prime=T_Y^\ast X$ and denote by $\tau$ the projection $E'\to Y$.
Let $\shm\in\text{Mod}_{\text{sp}}(\shd_X)$.
Recall that the composition of $\shf$ with $\nu_Y$ gives an exact functor $\mu_Y$ from $\text{Mod}_{\text{sp}}(\shd_X)$ to
$\text{Mod}_{\text{mon}}(\shd_{[T_Y^\ast X]})$, the microlocalization along $Y$ (cf~\cite{MF} for details).

\subsection{Fourier transform and microlocalization for monodromic $\shd^{\hbar}$-modules.}
\hspace{1mm}

 Let $\mathcal{S}$ be the full Serre substack of $\mathscr{M}{od}_{coh}(\shd_{[E]}^{\hbar})$:  $$U\mapsto \mathcal{S}(U)=\cup_{n\geq 0} \text{Mod}_{mon}(\shd_{[E],n}^{\hbar}|_U)$$
where we consider the structure of $\shd_{[E]}$-locally free module on $\shd_{[E],n}^{\hbar}$.
We shall denote by $\text{Mod}_{mon}(\shd_{[E]}^{\hbar})$ the category $\text{Mod}_{\mathcal{S}}(\shd_{[E]}^{\hbar}).$

Similarly we denote by $\mathcal{S}'$ the full Serre substack $$V\mapsto\mathcal{S}'(V)=\cup_{n\geq 0} \text{Mod}_{mon}(\shd_{[E'],n}^{\hbar}|_V).$$

Consider the functor $\Phi: Op(E)\to Op(E')$ given by $U\mapsto \widetilde{\pi}^{-1}\pi(U)$.
Since $\shf(U):\mathcal{S}(U)\to \mathcal{S}(\Phi(U))$
is clearly an exact functor, we are in conditions to apply Theorem \ref{T:1} and extend it as an exact functor $$\shf^{\hbar}: \text{Mod}_{mon}(\shd_{[E]}^{\hbar})\to \text{Mod}_{mon}(\shd_{[E']}^{\hbar})$$ setting  $$\shf^{\hbar}(\shn):= \underset{{n\geq 0}}\varprojlim \shf(\shn_n)=\underset{{n\geq 0}}\varprojlim \Omega_{E/Y}\otimes_{\pi^{-1}\sho_Y}\shn_n.$$

\begin{definition}\label{Fourier}
We define $\shf^{\hbar}$ as the  Fourier transform for $\shd^{\hbar}_{\left[E\right]}$-monodromic modules.
\end{definition}

In view of Definition \ref{D:T} and  Theorem \ref{T:1} we have functorial isomorphisms:
$$\shf^{\hbar}(\shn)\simeq
 \Omega_{E|Y} \otimes_{\pi^{-1}\sho_Y} \shn$$
  for   $\shn\in\text{Mod}_{mon}(\shd_{[E]}^{\hbar})$.

One can restate Propositions~\ref{G20},~\ref{G18}, Theorem~\ref{T:300} and Corollaries~\ref{C171} and~\ref{C173} replacing the functor $F^{\hbar}$  by $\shf^{\hbar}$.

\begin{lemma}\label{L123}
\begin{enumerate}[{\rm (a)}]
\item
Let $\shn$ be a monodromic $\shd_{\left[E\right]}^\hbar$-module.
Then $\shf^\hbar(\shn)$ is a monodromic $\shd_{\left[E^\prime\right]}^\hbar$-module.
\item
Let $\shm$ be a specializable $\shd_X^\hbar$-module.
Then $\nu_Y^\hbar(\shm)$ is a monodromic $\shd_{T_YX}^\hbar$-module.
\end{enumerate}
\end{lemma}
\begin{proof}
The first property results from the formulas $\shf^\hbar(\shn)_0\simeq \shf(\shn_0)$ and
$_0\shf^\hbar(\shn)\simeq \shf(_0\shn)$.

The second one is a consequence of the same property for $\shd_X$-modules and of the formulas

$$\nu_Y^\hbar(\shm)_0\simeq\nu_Y(\shm_0),\,\nu_Y^\hbar(\shm)\simeq \nu_Y({}_0\shm).$$
\end{proof}

\begin{definition}\label{D61}
The functor of microlocalization for $\shd_X^\hbar$-modules along a  submanifold $Y$ is given by:
\begin{eqnarray*}
\mu_Y^\hbar:\text{Mod}_{\text{sp}}(\shd_X^\hbar)&\to&\text{Mod}_\text{mon}(\shd_{[T_Y^\ast X]}^\hbar) \\
\shm&\mapsto&\mu_Y^\hbar(\shm)\mathbin{:=} \shf^\hbar( \nu_Y^\hbar(\shm)).
\end{eqnarray*}
\end{definition}
Since all the functors involved are exact and take values in subcategories of coherent modules, we conclude functorial isomorphisms:
\begin{eqnarray*}
\mu_Y^\hbar(\shm) \simeq
\underset{{n\geq 0}}\varprojlim\mu_Y(\shm_n),
\end{eqnarray*}
for $\shm\in\text{Mod}_{\text{sp}}(\shd_X^\hbar)$.

Smilarly we conclude:
\begin{theorem}\label{T60}
 We have natural  isomorphisms in $D^b(\C^\hbar_{T_Y^\ast X})$:
$$\begin{cases}
\shf(Sol_\hbar(\shn))\simeq Sol_\hbar(\shf^\hbar(\shn))[{-codim} Y]\\
\shf(DR_\hbar(\shn))\simeq DR_\hbar(\shf^\hbar(\shn))[{-codim} Y]
\end{cases}$$
for $\shn\in{ Mod}_{ mon}(\shd_{T_YX}^\hbar)$.
\end{theorem}
\begin{corollary}\label{C61}
 We have natural isomorphisms in $D^b(\C^\hbar_{T_Y^\ast X})$:
$$\begin{cases}
Sol_\hbar(\mu_Y^\hbar(\shm))\simeq \mu_Y(Sol_\hbar(\shm))[{ codim} Y]\\
DR_\hbar(\mu_Y^\hbar(\shm))\simeq \mu_Y(DR_\hbar(\shm))[{ codim} Y]
\end{cases}$$
for $\shm\in\text{Mod}_{\text{rh}}(\shd_X^\hbar)$.
\end{corollary}


\end{document}